\documentclass{article}
\usepackage{a4}
\usepackage{paralist}
\usepackage{graphicx, psfrag}
\usepackage{pgfplotstable}
\usepackage{subcaption}

\usepackage{amsmath}
\usepackage{amsthm}
\usepackage{amsfonts}
\usepackage{booktabs}
\usepackage{graphicx}
\usepackage{diagbox}

\usepackage{multirow}
\usepackage{hyperref}

\usepackage[ulem=normalem,draft]{changes}
\title{State-based approach to the numerical solution of Dirichlet boundary 
              optimal control problems for the Laplace equation}

\author{Ulrich~Langer\footnote{Institute of Computational Mathematics,
    Johannes Kepler University Linz, 
    and Johann Radon Institute for Computational and Applied Mathematics,
    Austrian Academy of Sciences,
    Altenberger Stra{\ss}e 69, 4040 Linz,
    Austria, Email: ulanger@numa.uni-linz.ac.at},
  \;
  Richard~L\"oscher\footnote{Institut f\"{u}r Angewandte Mathematik,
    Technische Universit\"{a}t Graz, Steyrergasse 30, 8010 Graz, Austria,
    Email: loescher@math.tugraz.at}, 
  \; Olaf~Steinbach\footnote{Institut f\"{u}r Angewandte Mathematik,
    Technische Universit\"{a}t Graz, Steyrergasse 30, 8010 Graz, Austria,
    Email: o.steinbach@tugraz.at}, 
  \; Huidong~Yang\footnote{Faculty of Mathematics, University of Vienna,
  and 
  Doppler Laboratory for Mathematical Modeling and Simulation of Next
  Generations of Ultrasound Devices (MaMSi),
  Oskar--Morgenstern--Platz 1, A-1090 Wien, Austria, Email: huidong.yang@univie.ac.at}
}  

\date{}

\newcommand{\norm}[1]{\|#1\|}

\pgfplotsset{compat=1.17}

\newtheorem{theorem}{Theorem}
\newtheorem{lemma}{Lemma}

\newtheorem{remark}{Remark}

\numberwithin{equation}{section} 

\begin{document}

\maketitle

\begin{abstract}
We investigate the Dirichlet boundary control of the Laplace equation,
considering the control in $H^{1/2}(\partial \Omega)$, which is the natural
space for Dirichlet data when the state belongs to $H^1(\Omega)$. The cost
of the control is measured in the $H^{1/2}(\partial \Omega)$ norm that also
plays the role of the regularization term. We discuss regularization and
finite element error estimates enabling us to derive an optimal relation
between the finite element mesh size $h$ and the regularization parameter
$\varrho$, balancing the energy cost for the control and the accuracy of
the approximation of the desired state. This relationship is also crucial in
designing efficient solvers. We also discuss additional  box constraints
imposed on the control and the state. Our theoretical findings are
complemented by numerical examples, including one example with box constraints.
\end{abstract} 

\noindent
\begin{keywords}
  Dirichlet boundary control problems, Laplace equation,
  finite element discretization, error estimates, solution methods.
\end{keywords}

\medskip

\noindent
\begin{msc}
49J20, 
49K20,  
65K10, 
65N30, 
65N22  
\end{msc}
\section{Introduction}\label{Sec:Introduction}
We consider the boundary optimal control problem to minimize the cost
functional
\begin{equation}\label{Eqn:Functional abstract}
  {\mathcal{J}}(y_\varrho,u_\varrho) =
  \frac{1}{2} \, \| y_\varrho - \overline{y} \|^2_{L^2(\Omega)} +
  \frac{1}{2} \, \varrho \, \| u_\varrho \|^2_U
\end{equation}
subject to the Dirichlet boundary value problem for the Laplace equation,
\begin{equation}\label{Eqn:Laplace DBVP}
  - \Delta y_\varrho = 0 \quad \mbox{in} \; \Omega, \quad
  y_\varrho = u_\varrho \quad \mbox{on} \; \Gamma = \partial \Omega .
\end{equation}
Here, $\Omega \subset {\mathbb{R}}^d$, $d=2,3$, is a bounded domain with
Lipschitz boundary $\Gamma = \partial \Omega$,
$\overline{y} \in L^2(\Omega)$ is a given target, $\varrho \in (0,1]$ is
some regularization or cost parameter which the solution depends on, and
$U$ is the control space.

The mathematical analysis of abstract optimal control problems goes back
to the pioneering work of Lions \cite{Lions1971}, see also the more recent
text books \cite{HinzePinnauUlbrichUlbrich2009,Troeltzsch2010}.
Dirichlet boundary control problems play an important role in the context
of computational fluid mechanics, see, e.g.,
\cite{Fursikov1998,Gunzburger1992,HinzeKunisch2004}; for related work
for parabolic evolution equations, see, e.g.,
\cite{AradaRaymond2002,KunischVexler2007SISC,Lasiecka1978,Liang2025}.
In \cite{KunischVexler2007SISC}, the authors give an overview on
different approaches to handle Dirichlet boundary control problems
by considering different control spaces, including
$U=L^2(\Gamma)$, $U=H^{1/2}(\Gamma)$, and $U=H^1(\Gamma)$. The latter
was used to avoid difficulties with implementing the $H^{1/2}$-norm
which, at first glance, seems to be
more involved to realize than the $L^2$-norm. In fact, the most
standard choice in Dirichlet boundary control problems is to consider
$U=L^2(\Gamma)$ where the Dirichlet boundary value problem
\eqref{Eqn:Laplace DBVP} is formulated in a very weak sense
\cite{Berggren2004}, see, e.g.,
\cite{ApelMateosPfeffererRoesch2015SICO,Casas1993SICON,CasasRaymond2006SICO,
DeckelnickGuentherHinze2009SICO,MayRannacherVexler2013SICO}.
However, when considering the Dirichlet boundary value problem
\eqref{Eqn:Laplace DBVP}, $U = H^{1/2}(\Gamma)$ appears as a
natural choice \cite{BenBelgacem2003}. In 
\cite{OfPhanSteinbach2015NumerMath}, we have considered a finite element
formulation for the solution of Dirichlet boundary control problems when
using the regularization norm induced by the Steklov--Poincar\'e operator
$S : H^{1/2}(\Gamma) \to H^{-1/2}(\Gamma)$, see also
\cite{ChowdhuryGudiNandakumaran2017MathComp}. Since the Steklov--Poincar\'e
operator computes the normal derivative of the harmonic extension of a
given Dirichlet datum, this norm representation is equal to the 
Dirichlet norm as used in the variational formulation of Neumann
boundary value problems for the Laplacian. In this case, the abstract
theory as presented recently in
\cite{LangerLoescherSteinbachYang2025arXiv2505.19062}
is applicable, as we will see in this paper. In particular, if the target
$\overline{y}$ itself is harmonic, we can derive related regularization
error estimates which depend on the Sobolev regularity of the
target $\overline{y}$, and on the smoothness of the domain.
In the more general situation when the target $\overline{y}$ is not harmonic,
we determine the best approximation of the target in the space of harmonic
functions. For the finite element discretization of the gradient equation
in the space of harmonic functions, we apply a non-conforming mixed approach,
for which we prove related finite element error estimates and establish
optimal relations between the regularization parameter $\varrho$ and the
finite element mesh size $h$. This optimal choice also enables an
efficient iterative solution of the discrete system which is related to the
first biharmonic boundary value problem. Hence we can apply well known
preconditioning techniques, e.g.,
\cite{BraessPeisker:1986IMAJNA,GlowinskiPironneau:1979SIAMRev,
  Langer1986NumerMath,Peisker1988M2AN}. However, the relation of
Dirichlet boundary control problems with the first biharmonic boundary
value problem has some serious implications when considering domains
with a piecewise smooth boundary: In the case of the control space
$U=H^{1/2}(\Gamma)$, we have to adapt the optimal choice of the
regularization parameter $\varrho$ as a function of the finite element
mesh size $h$ by some additional logarithmic terms.
When using $U = L^2(\Gamma)$, we even have a reduced regularity of the
optimal control which, in the case of a two-dimensional domain with
piecewise smooth boundary, is zero in all corner points, see
\cite[Remark 1]{MayRannacherVexler2013SICO} and
\cite[Figure 3]{OfPhanSteinbach2015NumerMath}.

The rest of the paper is organized as follows. Section~\ref{Sec:BOCP}
starts with the setting of the Dirichlet boundary optimal 
control problem for the Laplace equation with energy regularization. 
Furthermore, we derive estimates of the deviation of the state $y_\varrho$ 
from the given desired state $\overline{y}$ in the $L^2$ norm as well as 
estimates of the energy cost of the control in terms of $\varrho$ 
and in dependence of the regularity of the desired state $\overline{y}$
and the computational domain $\Omega$. 
The finite element discretization together with the corresponding 
discretization error analysis is 
performed in Section~\ref{Sec:FED}. It turns out that the cost parameter $\varrho$
must be balanced with the mesh size $h$ in order to obtain asymptotically 
optimal estimates under minimal energy costs.
Section~\ref{Sec:Solver} is devoted to the efficient solution 
of the resulting mixed finite element scheme. 
In Section~\ref{Sec:Constraints}, we briefly show that control and state constraints
can easily be incorporated into our framework. 
The numerical results presented in Section~\ref{Sec:NumericalResults} 
illustrate the theoretical findings quantitatively.
We tested our approach for 2D and 3D benchmark problems with different
features. 
Finally, we draw some conclusions, and give an outlook on further topics.

%
%
\section{Boundary control of the Laplace equation}
\label{Sec:BOCP}
In the case of energy regularization with the control space $U=H^{1/2}(\Gamma)$, 
we consider the boundary optimal
control problem to minimize the cost functional
\begin{equation}\label{Eqn:Functional}
  {\mathcal{J}}(y_\varrho,u_\varrho) =
  \frac{1}{2} \, \| y_\varrho - \overline{y} \|^2_{L^2(\Omega)} +
  \frac{1}{2} \, \varrho \, | u_\varrho |^2_{H^{1/2}(\Gamma)}
\end{equation}
subject to the Dirichlet boundary value problem \eqref{Eqn:Laplace DBVP}.
Hence, we define the state space
\[
  Y := \Big \{ y \in H^1(\Omega) :
  \langle \nabla y , \nabla v \rangle_{L^2(\Omega)} = 0 \;\,
  \forall v \in H^1_0(\Omega) \Big \} 
\]
of all harmonic functions. The state to
control map is the interior Dirichlet trace operator
$\gamma_0^{int} : Y \to H^{1/2}(\Gamma)$ which is an isomorphism. Once the
optimal state $y_\varrho \in Y$ is known, the related optimal control 
$u_\varrho = \gamma_0^{int} y_\varrho$ is nothing but the trace of $y_\varrho$
on $\Gamma$. The semi-norm in $H^{1/2}(\Gamma)$ is induced by the
Steklov--Poincar\'e operator, e.g.,
\cite{Agoshkov1985,QuarteroniValli1999,Steinbach:2008NumericalApprox},
${\mathcal{S}} : H^{1/2}(\Gamma) \to H^{-1/2}(\Gamma)$, defined as
\[
  |u_\varrho|^2_{H^{1/2}(\Gamma)} = \langle
  {\mathcal{S}} u_\varrho , u_\varrho \rangle_\Gamma =
  \int_\Gamma \frac{\partial}{\partial n_x} y_\varrho(x) \, u_\varrho(x) \, dx =
  \int_\Omega |\nabla y_\varrho(x)|^2 \, dx,
\]
where $y_\varrho \in Y$ is the harmonic extension of $u_\varrho \in H^{1/2}(\Gamma)$. 
Instead of \eqref{Eqn:Functional}, we now
consider the reduced state-based cost functional
\begin{equation}\label{Eqn:reduced functional}
  \widetilde{\mathcal{J}}(y_\varrho) =
  \frac{1}{2} \, \| y_\varrho - \overline{y} \|^2_{L^2(\Omega)} +
  \frac{1}{2} \, \varrho \, \| \nabla y_\varrho \|^2_{L^2(\Omega)} \quad
  \mbox{for} \; y_\varrho \in Y.
\end{equation}
The minimizer of the reduced cost functional \eqref{Eqn:reduced functional}
is given as the unique solution $y_\varrho \in Y$ of the gradient equation
such that
\begin{equation}\label{Eqn:Gradient Equation}
  a(y_\varrho,y) := \int_\Omega y_\varrho(x) \, y(x) \, dx +
  \varrho \int_\Omega \nabla y_\varrho(x) \cdot \nabla y(x) \, dx
  \, = \,
  \int_\Omega \overline{y}(x) \, y(x) \, dx
\end{equation}
is satisfied for all $y \in Y$. When introducing the weighted norm
\[
  \| y \|_Y := \left( \int_\Omega [y(x)]^2 \, dx + \varrho \int_\Omega 
    |\nabla y(x)|^2 \, dx \right)^{1/2},
\]
we immediately conclude ellipticity and boundedness of the bilinear form
$a(\cdot,\cdot)$,
\begin{equation}\label{Eqn:Properties A}
  a(y,y) = \| y \|^2_Y, \quad
  |a(y,v)| \leq \| y \|_Y \| v \|_Y \quad
  \mbox{for all} \; y,v \in Y .
\end{equation}
The variational formulation \eqref{Eqn:Gradient Equation} corresponds to the
abstract setting as considered in
\cite[Lemma 1]{LangerLoescherSteinbachYang2025arXiv2505.19062}.
Hence, we have the following results: 
For $\overline{y} \in L^2(\Omega)$, there holds
\begin{equation}\label{Eqn:Regularization error L2}
  \| y_\varrho - \overline{y} \|_{L^2(\Omega)} \leq
  \| \overline{y} \|_{L^2(\Omega)}, 
\end{equation}
as well as
\begin{equation}\label{Eqn:Bound L2}
  \| y_\varrho \|_{L^2(\Omega)} \leq
  \| \overline{y} \|_{L^2(\Omega)}, \quad
  \| \nabla y_\varrho \|_{L^2(\Omega)} \leq \varrho^{-1/2} \,
  \| \overline{y} \|_{L^2(\Omega)} .
\end{equation}
For $\overline{y} \in Y$, we have
\begin{equation}\label{Eqn:Regularization error Y}
  \| y_\varrho - \overline{y} \|_{L^2(\Omega)} \leq
  \varrho^{1/2} \, \| \nabla \overline{y} \|_{L^2(\Omega)}, \quad
  \| \nabla (y_\varrho - \overline{y}) \|_{L^2(\Omega)} \leq
  \| \nabla \overline{y} \|_{L^2(\Omega)}, 
\end{equation}
as well as
\begin{equation}\label{Eqn:Bound Y}
  \| \nabla y_\varrho \|_{L^2(\Omega)} \leq
  \| \nabla \overline{y} \|_{L^2(\Omega)} .
\end{equation}
While the regularization error estimates
\eqref{Eqn:Regularization error L2}--\eqref{Eqn:Bound Y} hold for
Lipschitz domains $\Omega$, for more regular targets
$\overline{y}$ we need to formulate additional assumptions on $\Omega$:

\begin{lemma}\label{Lemma:Regularization error smooth}
  Let $ y_\varrho \in Y$ be the unique solution of the variational formulation
  \eqref{Eqn:Gradient Equation}. Assume that $\Omega \subset {\mathbb{R}}^n$,
  $n=2,3$, has a smooth boundary $\Gamma$, and
  $\overline{y} \in Y \cap H^2(\Omega)$. Then,
  \begin{equation}\label{Eqn:Regularization error regular smooth}
    \| y_\varrho - \overline{y} \|_{L^2(\Omega)} \leq
    c \, \varrho \, \| \overline{y} \|_{H^2(\Omega)},
    \quad
    \| \nabla (y_\varrho - \overline{y}) \|_{L^2(\Omega)}
    \leq
    c \, \varrho^{1/2} \, \| \overline{y} \|_{H^2(\Omega)} .
  \end{equation}
\end{lemma}

\begin{proof}
  When considering \eqref{Eqn:Gradient Equation} for
  $y = y_\varrho - \overline{y} \in Y$, this gives
  \begin{equation}\label{Lemma 1 Step 1}
    \| y_\varrho - \overline{y} \|^2_{L^2(\Omega)}
    + \varrho \, \| \nabla (y_\varrho - \overline{y}) \|^2_{L^2(\Omega)}   
    =
    \varrho \, \langle \nabla \overline{y} ,
    \nabla (\overline{y} - y_\varrho) \rangle_{L^2(\Omega)} .
  \end{equation}
  Now, using integration by parts, duality, and the trace theorem, we obtain,
  since $\overline{y} \in Y$ is assumed to be harmonic,
  \begin{eqnarray*}
    \langle \nabla \overline{y} ,
    \nabla (\overline{y} - y_\varrho) \rangle_{L^2(\Omega)}
    & = & \langle \partial_n \overline{y} , \overline{y} - y_\varrho
          \rangle_{L^2(\Gamma)} \\[1mm]
    && \hspace*{-3cm}
       \leq \, \| \partial_n \overline{y} \|_{H^{1/2}(\Gamma)}
       \| \overline{y} - y_\varrho \|_{H^{-1/2}(\Gamma)} \, \leq \,
       c \, \| \overline{y} \|_{H^2(\Omega)}
       \| \overline{y} - y_\varrho \|_{H^{-1/2}(\Gamma)} .
  \end{eqnarray*}
  It remains to consider
  \[
    \| \overline{y} - y_\varrho \|_{H^{-1/2}(\Gamma)}
    = \sup\limits_{0 \neq \varphi \in H^{1/2}(\Gamma)}
    \frac{\langle \overline{y}-y_\varrho , \varphi
      \rangle_\Gamma}{\| \varphi \|_{H^{1/2}(\Gamma)}},
  \]
  where $\langle \cdot , \cdot \rangle_\Gamma$ is the duality pairing
  as extension of the inner product in $L^2(\Gamma)$.
  For any $\varphi \in H^{1/2}(\Gamma)$ we determine
  $\psi_\varphi \in H^2(\Omega)$ as weak solution of
  the first boundary value problem for the biharmonic equation,
  \begin{equation}\label{biharmonic equation}
    \Delta^2 \psi_\varphi = 0 \quad \mbox{in} \; \Omega, \quad
    \psi_\varphi = 0, \quad
    \partial_n \psi_\varphi = \varphi \quad
    \mbox{on} \; \Gamma .
  \end{equation}
  With this we can compute, using Green's formula, recall
  $\overline{y} - y_\varrho \in Y$ being harmonic,
  \begin{eqnarray*}
  \langle \overline{y}-y_\varrho , \varphi \rangle_\Gamma
    & = & \int_\Gamma [\overline{y} - y_\varrho ] \,
          \partial_n \psi_\varphi \, ds_x \\
    && \hspace*{-2cm} = \, \int_\Omega \nabla \psi_\varphi \cdot \nabla
          [\overline{y}-y_\varrho] \, dx +
          \int_\Omega \Delta \psi_\varphi \,
          [\overline{y}-y_\varrho] \, dx \\
    && \hspace*{-2cm} = \, \int_\Gamma \frac{\partial}{\partial n_x}
          [\overline{y}-y_\varrho] \, \psi_\varphi \, ds_x
          + \int_\Omega [- \Delta (\overline{y}-y_\varrho)] \,
          \psi_\varphi \, dx +
          \int_\Omega \Delta \psi_\varphi \, [\overline{y}-y_\varrho] \, dx \\
    && \hspace*{-2cm}
       = \, \int_\Omega \Delta \psi_\varphi \, [\overline{y}-y_\varrho] \, dx
       \, \leq \, \| \Delta \psi_\varphi \|_{L^2(\Omega)}
       \| \overline{y}-y_\varrho \|_{L^2(\Omega)} .
  \end{eqnarray*}
  Now, using the regularity for the solution $\psi_\varphi$
  of the first boundary
  value problem for the biharmonic equation, e.g.,
  \cite[Subsection 2.4]{GlowinskiPironneau:1979SIAMRev},
  \[
    \| \Delta \psi \|_{L^2(\Omega)} \leq
    \| \psi \|_{H^2(\Omega)} \leq c \,
    \| \partial_n \psi \|_{H^{1/2}(\Gamma)} = c\,
    \| \varphi \|_{H^{1/2}(\Gamma)} ,
  \]
  we finally conclude
  \[
    \| \overline{y} - y_\varrho \|_{H^{-1/2}(\Gamma)} \leq c \,
    \| \overline{y} - y_\varrho \|_{L^2(\Omega)} .
  \]
  Hence, we obtain
  \[
    \| y_\varrho - \overline{y} \|^2_{L^2(\Omega)}
    + \varrho \, \| \nabla (y_\varrho - \overline{y}) \|^2_{L^2(\Omega)}   
    \, \leq \, c \, \varrho \, \| \overline{y} \|_{H^2(\Omega)}
    \| \overline{y} - y_\varrho \|_{L^2(\Omega)},
  \]
  and the regularization error estimates
  \eqref{Eqn:Regularization error regular smooth} follow.
\end{proof}

\noindent
When the boundary $\Gamma$ is not sufficiently smooth, the
normal derivative $\partial_n \psi$
even for smooth functions $\psi$ becomes discontinuous, and the
proof of Lemma \ref{Lemma:Regularization error smooth} is not valid anymore,
i.e., the first boundary value problem for the biharmonic equation
\eqref{biharmonic equation} is not well defined.
Therefore, we now consider the case of a piecewise smooth boundary
\begin{equation}\label{pw boundary}
  \Gamma = \bigcup\limits_{j=1}^J \overline{\Gamma}_j, \quad
  \Gamma_i \cap \Gamma_j = \emptyset \quad
  \mbox{for} \; j \neq i ,
\end{equation}
and we define the local Sobolev spaces, e.g., \cite{McLean2000},
$H^{1/2}(\Gamma_j)$, $\widetilde{H}^{1/2}(\Gamma_j) = H^{1/2}_{00}(\Gamma_j)$,
with their respective dual spaces
$\widetilde{H}^{-1/2}(\Gamma_j) = [H^{1/2}(\Gamma_j)]^*$,
$H^{-1/2}(\Gamma_j) = [\widetilde{H}^{1/2}(\Gamma_j)]^*$.
At this time, we will restrict to the case where
$\Omega\subset \mathbb{R}^2$.

\begin{lemma}\label{Lemma:Regularization error}
  Let $ y_\varrho \in Y$ be the unique solution of the variational formulation
  \eqref{Eqn:Gradient Equation}. Assume that $\Omega \subset {\mathbb{R}}^2$
  is convex with a piecewise smooth boundary $\Gamma$ as given in
  \eqref{pw boundary}, and assume $\overline{y} \in Y \cap H^2(\Omega)$. Then,
  \begin{equation}\label{Eqn:Regularization error pw 1}
    \| y_\varrho - \overline{y} \|_{L^2(\Omega)} \leq
    c \, \varrho \, (1 + |\log \varrho|) \, \| \overline{y} \|_{H^2(\Omega)},
  \end{equation}
  and
  \begin{equation}\label{Eqn:Regularization error pw 2}
    \| \nabla (y_\varrho - \overline{y}) \|_{L^2(\Omega)}
    \leq
    c \, \varrho^{1/2} \, (1+|\log \varrho|) \, \| \overline{y} \|_{H^2(\Omega)} .
  \end{equation}
\end{lemma}
\begin{proof}
  We first proceed as in the proof of
  Lemma \ref{Lemma:Regularization error smooth}, i.e., we have
  \eqref{Lemma 1 Step 1}. But now we have to consider
  \begin{eqnarray*}
    \langle \nabla \overline{y} ,
    \nabla (\overline{y} - y_\varrho) \rangle_{L^2(\Omega)}
    & = & \langle \partial_n \overline{y} , \overline{y} - y_\varrho
          \rangle_{L^2(\Gamma)} \, = \, \sum\limits_{j=1}^J
          \langle \partial_n \overline{y} , \overline{y} - y_\varrho
          \rangle_{L^2(\Gamma_j)} \\
    & \leq & \sum\limits_{j=1}^J
             \| \partial_n \overline{y} \|_{H^{1/2}(\Gamma_j)}
             \| \overline{y} - y_\varrho \|_{\widetilde{H}^{-1/2}(\Gamma_j)} \\
    & \leq & \left( \sum\limits_{j=1}^J
             \| \partial_n \overline{y} \|^2_{H^{1/2}(\Gamma_j)} \right)^{1/2}
             \left( \sum\limits_{j=1}^J
             \| \overline{y} - y_\varrho \|^2_{\widetilde{H}^{-1/2}(\Gamma_j)}
             \right)^{1/2} \\
    & \leq & c \, \| \overline{y} \|_{H^2(\Omega)}
             \left( \sum\limits_{j=1}^J
             \| \overline{y} - y_\varrho \|^2_{\widetilde{H}^{-1/2}(\Gamma_j)}
             \right)^{1/2} \, .
  \end{eqnarray*}
  In order to proceed as in the proof of
  Lemma \ref{Lemma:Regularization error smooth}, we need to replace
  the local norms
  $\| \overline{y} - y_\varrho \|_{\widetilde{H}^{-1/2}(\Gamma_j)}$ by
  $\| \overline{y} - y_\varrho \|_{H^{-1/2}(\Gamma_j)}$. With respect to
  each $\Gamma_j$, we define local finite element spaces
  $V_j := S_\eta^1(\Gamma_j) \cap \widetilde{H}^{1/2}(\Gamma_j)$ with mesh
  size $\eta$ of piecewise
  linear continuous basis functions which are zero at the end points of
  $\Gamma_j$, and $W_j := \widetilde{S}_\eta^1(\Gamma)$ is the space of
  piecewise linear functions which are constant at those elements touching
  an end point of $\Gamma_j$, such that
  $\mbox{dim} \, V_j = \mbox{dim} \, W_j$.
  Then we define the generalized $L^2$ projection
  $\widetilde{Q}_j (\overline{y} - y_\varrho) \in W_j$ as unique solution
  of the variational formulation
  \[
    \langle \widetilde{Q}_j (\overline{y} - y_\varrho) ,
    \phi_j \rangle_{L^2(\Gamma_j)} =
    \langle \overline{y} - y_\varrho ,
    \phi_j \rangle_{L^2(\Gamma_j)} \quad \mbox{for all} \; \phi_j \in V_j .
  \]
  With this we can write, using standard approximation error estimates,
  e.g., \cite{Steinbach:2001NumerMath,Steinbach:2002NumerMath}, and the
  trace theorem,
  \begin{eqnarray*}
    \| \overline{y} - y_\varrho \|_{\widetilde{H}^{-1/2}(\Gamma_j)}
    & \leq & \| (I-\widetilde{Q}_j)(\overline{y} - y_\varrho)
             \|_{\widetilde{H}^{-1/2}(\Gamma_j)}
             +
             \| \widetilde{Q}_j (\overline{y} - y_\varrho)
             \|_{\widetilde{H}^{-1/2}(\Gamma_j)} \\
    & \leq & c \, \eta \, | \overline{y} - y_\varrho |_{H^{1/2}(\Gamma_j)}
             +
             \| \widetilde{Q}_j (\overline{y} - y_\varrho)
             \|_{\widetilde{H}^{-1/2}(\Gamma_j)} \\
    & \leq & c \, \eta \, \| \nabla (\overline{y} - y_\varrho) \|_{L^2(\Omega)}
             +
             \| \widetilde{Q}_j (\overline{y} - y_\varrho)
             \|_{\widetilde{H}^{-1/2}(\Gamma_j)} .
  \end{eqnarray*}
  It remains to consider
  \[
    \| \widetilde{Q}_j (\overline{y} - y_\varrho)
    \|_{\widetilde{H}^{-1/2}(\Gamma_j)} =
    \sup\limits_{0 \neq \varphi \in H^{1/2}(\Gamma_j)}
    \frac{\langle \widetilde{Q}_j (\overline{y} - y_\varrho),
      \varphi \rangle_{\Gamma_j}}{\| \varphi \|_{H^{1/2}(\Gamma_j)}} \, .
  \]
  Let $\widetilde{Q}_j^* \varphi \in V_j$ be the unique solution
  of the variational formulation
  \[
    \langle \widetilde{Q}_j^* \varphi , \psi_j \rangle_{L^2(\Gamma_j)} =
    \langle \varphi , \psi_j \rangle_{L^2(\Gamma_j)} \quad
    \mbox{for all} \; \psi_j \in W_j,
  \]
  and we have
  \[
    \langle \widetilde{Q}_j (\overline{y} - y_\varrho),
    \varphi \rangle_{\Gamma_j} =
    \langle \widetilde{Q}_j (\overline{y}-y_\varrho) ,
    \widetilde{Q}_j^* \varphi \rangle_{L^2(\Gamma_j)} =
    \langle \overline{y}-y_\varrho ,
    \widetilde{Q}_j^* \varphi \rangle_{L^2(\Gamma_j)} .
  \]
  With the stability estimate
  \cite{Steinbach:2001NumerMath,Steinbach:2002NumerMath}
  \[
    \| \widetilde{Q}_j^* \varphi \|_{H^{1/2}(\Gamma_j)} \leq c \,
    \| \varphi \|_{H^{1/2}(\Gamma_j)} ,
  \]
  and using, since $\widetilde{Q}_j^* \varphi \in V_j \subset
  \widetilde{H}^{1/2}(\Gamma)$, the estimate
  \cite[Theorem 4.1]{McLeanSteinbach:1999AdvComputMath}
  \begin{equation}\label{Eqn:Einettung Htilde}
    \| \widetilde{Q}_j^* \varphi \|_{\widetilde{H}^{1/2}(\Gamma_j)}
    \leq c \, (1 + |\log \eta|) \,
    \| \widetilde{Q}_j^* \varphi \|_{H^{1/2}(\Gamma_j)} ,
  \end{equation}
  we further conclude
  \begin{eqnarray*}
    \frac{c}{1 + |\log \eta|} \,
    \| \widetilde{Q}_j (\overline{y} - y_\varrho)
    \|_{\widetilde{H}^{-1/2}(\Gamma_j)}
    & \leq & \sup\limits_{0 \neq \varphi \in H^{1/2}(\Gamma_j)}
          \frac{\langle \overline{y} - y_\varrho, \widetilde{Q}_j^*
          \varphi \rangle_{\Gamma_j}}
          {\| \widetilde{Q}_j^* \varphi \|_{\widetilde{H}^{1/2}(\Gamma_j)}}
    \\[2mm]
    & \leq & \| \overline{y} - y_\varrho \|_{H^{-1/2}(\Gamma_j)} \, .
  \end{eqnarray*}
  Hence we can proceed as in the proof of
  Lemma \ref{Lemma:Regularization error smooth}, i.e.,
  we consider the first boundary value problem for the biharmonic equation
  \eqref{biharmonic equation} for functions $\varphi \in H^{1/2}(\Gamma)$
  which are zero in all corner points. Then, all arguments as used
  in the proof of  Lemma \ref{Lemma:Regularization error smooth}
  remain valid, see \cite{ZhangXu:2014SINUM}.
  From \eqref{Lemma 1 Step 1} we therefore conclude
  \begin{eqnarray*}
    && \| y_\varrho - \overline{y} \|^2_{L^2(\Omega)}
       +
       \varrho \, \| \nabla (y_\varrho - \overline{y}) \|^2_{L^2(\Omega)} \\
    && \hspace*{1cm} \leq c \, \varrho \, \| \overline{y} \|_{H^2(\Omega)}
       \left( \eta^2 \, \| \nabla (\overline{y} - y_\varrho) \|^2_{L^2(\Omega)}
       +
       (1+|\log \eta|)^2 \, \| \overline{y} - y_\varrho \|_{L^2(\Omega)}^2
       \right)^{1/2} .
  \end{eqnarray*}
  From this we obtain
  \begin{eqnarray*}
    && \| \nabla (y_\varrho - \overline{y}) \|^4_{L^2(\Omega)} 
       - c \, \eta^2 \, \| \overline{y} \|_{H^2(\Omega)}^2
       \| \nabla (\overline{y} - y_\varrho) \|^2_{L^2(\Omega)} \\[1mm]
    && \hspace*{3cm}
       -
       c \, (1+|\log \eta|)^2 \, \| \overline{y} \|_{H^2(\Omega)}^2
       \| \overline{y} - y_\varrho \|_{L^2(\Omega)}^2
       \, \leq \, 0,
  \end{eqnarray*}
  i.e.
  \begin{equation}\label{Lemma 2 Step 1}
    \| \nabla (y_\varrho - \overline{y}) \|^2_{L^2(\Omega)} \leq
    c \, \eta^2 \, \| \overline{y} \|_{H^2(\Omega)}^2 +
    c \, (1+|\log \eta|) \, \| \overline{y} \|_{H^2(\Omega)}
    \| \overline{y} - y_\varrho \|_{L^2(\Omega)}.
  \end{equation}
  With this, we further have
   \begin{eqnarray*}
     && \hspace*{-7mm} \| y_\varrho - \overline{y} \|^4_{L^2(\Omega)}
     \leq \, c^3 \, \varrho^2 \, \eta^4 \, \| \overline{y} \|_{H^2(\Omega)}^4 +
              c^3 \, (1+|\log \eta|) \, \varrho^2 \, \eta^2 \,
              \| \overline{y} \|^3_{H^2(\Omega)}
              \| \overline{y} - y_\varrho \|_{L^2(\Omega)} \\[1mm]
     && \hspace*{5.2cm} +
        c^2 \, (1+|\log \eta|)^2 \, \varrho^2 \,
        \| \overline{y} \|_{H^2(\Omega)}^2
        \| \overline{y} - y_\varrho \|_{L^2(\Omega)}^2 .
  \end{eqnarray*}
  When introducing
  $Q := \| y_\varrho - \overline{y} \|_{L^2(\Omega)}/\|
  \overline{y} \|_{H^2(\Omega)}$, this can be written as
  \begin{eqnarray*}
    Q^4
    & \leq & c^3 \, \varrho^2 \, \eta^4 +
             c^3 \, (1+|\log \eta|) \, \eta^2 \, \varrho^2 \, Q +
             c^2 \, \varrho^2 \, (1+|\log \eta|)^2 \, Q^2 \\
    & \leq & c \, ( \varrho (1+|\log \eta|) Q + \varrho \, \eta^2)^2,
  \end{eqnarray*}
  i.e.,
  \[
    Q^2 \leq c \, \Big[ \varrho (1+|\log \eta|) Q + \varrho \eta^2 \Big] . 
  \]
  From this we find
  \[
    Q \leq c \Big[ \varrho (1+|\log \eta|) + \sqrt{\varrho} \eta\Big].
  \]
  When using $\eta = \varrho^\alpha$ the right hand side becomes minimal
  for $\alpha = 1/2$, i.e.,
  \[
    Q \leq c \, \varrho \, (1+|\log \varrho|) ,
  \]
  and we have proven \eqref{Eqn:Regularization error pw 1}.
  When inserting this result into \eqref{Lemma 2 Step 1}, this gives
  \[
    \| \nabla (y_\varrho - \overline{y}) \|^2_{L^2(\Omega)} \leq
    c \, \varrho \, \| \overline{y} \|_{H^2(\Omega)}^2 +
    c \, \varrho \, (1+|\log \varrho|)^2 \, \| \overline{y} \|_{H^2(\Omega)}^2,
  \]
  i.e., \eqref{Eqn:Regularization error pw 2} follows. 
\end{proof}

\begin{remark}
  The restriction to two-dimensional domains $\Omega$ with piecewise smooth
  boundary is only due to the requirement in using the estimate
  \eqref{Eqn:Einettung Htilde}. While we can expect
  a similar result for piecewise smooth domains in 3D, a rigorous proof
  seems to be open. However, our numerical results indicate that we
  may use \eqref{Eqn:Einettung Htilde} even in the 3D case.
\end{remark}

\noindent
The regularization error estimates as given in Lemma
\ref{Lemma:Regularization error smooth} and \ref{Lemma:Regularization error} 
hold true if the target $\overline{y}$ is harmonic. In the following we
consider the situation if this condition is violated.

\begin{lemma}
  Let $\mathcal{H}(\Omega):= \{y\in L^2(\Omega):\, \Delta y=0\}$ be the
  space of harmonic functions in $L^2(\Omega)$, and $H_0^2(\Omega):=
  \{w\in H^2(\Omega):\, w =\partial_nw =0 \; \mbox{on} \; \Gamma \}$.
  Then, there holds the $L^2$-orthogonal splitting
  \begin{equation*}
    L^2(\Omega) = \mathcal{H}(\Omega) \oplus \Delta (H^2_0(\Omega)).
  \end{equation*}
\end{lemma}

\begin{proof}
  For any $\overline{y} \in L^2(\Omega)$ let
  $\varphi_{\overline{y}}\in H^1_0(\Omega)$ denote the unique solution of 
  \begin{equation*}
    -\Delta \varphi_{\overline{y}} = \overline{y} \text{ in } \Omega, \quad
    \varphi_{\overline{y}} = 0 \text{ on } \Gamma.  
  \end{equation*}
  Moreover, let $\psi_{\overline{y}} \in H^2(\Omega)$ be the unique solution of
  the first boundary value problem for the biharmonic equation,
  \begin{equation*}
    \Delta^2 \psi_{\overline{y}} = 0 \text{ in } \Omega, \quad
    \psi_{\overline{y}} =0, \quad \partial_n \psi_{\overline{y}} =
    \partial_n \varphi_{\overline{y}} \text{ on } \Gamma. 
  \end{equation*}
  With this we define 
  $w_{\overline{y}} := \varphi_{\overline{y}} - \psi_{\overline{y}} \in
  H^1_0(\Omega)$, satisfying  
  \begin{equation*}
    -\Delta w_{\overline{y}} =
    -\Delta \varphi_{\overline{y}} + \Delta \psi_{\overline{y}} =
    \overline{y} + \Delta \psi_{\overline{y}} \in L^2(\Omega).
  \end{equation*}
  Note, that the traces are  
  \begin{equation*}
    w_{\overline{y}} =
    \varphi_{\overline{y}} - \psi_{\overline{y}} = 0 \quad \text{and} \quad
    \partial_n w_{\overline{y}} =
    \partial_n \varphi_{\overline{y}} - \partial_n \psi_{\overline{y}} = 0
    \quad \text{on } \Gamma,
  \end{equation*}
  and therefore, by \cite[Cor. 4.2]{BehrndtMicheler2014JFA},
  $w_{\overline{y}} \in H_0^2(\Omega)$. Then
  \begin{equation*}
    \overline{y} = - \Delta \varphi_{\overline{y}} =
    - \Delta (\psi_{\overline{y}}+w_{\overline{y}}) =
    - \Delta \psi_{\overline{y}} - \Delta w_{\overline{y}} \in
    \mathcal{H}(\Omega) + \Delta(H_0^2(\Omega)).
  \end{equation*}
  Since
  $\Delta^2 \psi_{\overline{y}} =0$,
  we define $\overline{y}^\ast := - \Delta \psi_{\overline{y}} \in
  {\mathcal{H}}(\Omega)$, and
  $\overline{y}_0 := - \Delta w_{\overline{y}} \in \Delta(H_0^2(\Omega))$.
  The sum is direct, since if
  $\overline{y} \in \mathcal{H}(\Omega) \cap \Delta(H_0^2(\Omega))$, we have
  that $\overline{y} = \Delta w_{\overline{y}}$ with
  $w_{\overline{y}}\in H^2_0(\Omega)$, i.e.,
  \begin{equation*}
    \Delta^2 w_{\overline{y}} =
    \Delta \overline{y} = 0 \text{ in } \Omega, \quad
    w_{\overline{y}}=0,\quad \partial_nw_{\overline{y}} =0 \text{ on }\Gamma, 
  \end{equation*}
  implying $w_{\overline{y}}=0$, and therefore $\overline{y}=0$ follows.
  To check orthogonality, we compute, applying integration by parts twice
  and using $\overline{y}^\ast \in {\mathcal{H}}(\Omega)$,

  \begin{eqnarray*}
    0
    & = & \langle -\Delta \overline{y}^\ast,w_{\overline{y}}
          \rangle_{L^2(\Omega)} \\
    & = & \langle \overline{y}^\ast,-\Delta w_{\overline{y}}
          \rangle_{L^2(\Omega)} -
          \langle \partial_n\overline{y}^\ast, w_{\overline{y}}
          \rangle_\Gamma  +
          \langle \overline{y},\partial_n w_{\overline{y}}
          \rangle_\Gamma \, = \,
          \langle \overline{y}^\ast,\overline{y}_0 \rangle_{L^2(\Omega)}. 
  \end{eqnarray*}
  This concludes the proof.
\end{proof}

\noindent
Hence, we can split each given target $\overline{y} \in L^2(\Omega)$ into
$\overline{y} = \overline{y}^\ast + \overline{y}_0 \in
\mathcal{H}(\Omega)\oplus \Delta(H^2_0(\Omega))$. Using the $L^2$-orthogonality
of the splitting, we further have, for any
$y\in Y\subset \mathcal{H}(\Omega)$, that  
\begin{equation*}
  \int_\Omega \overline{y}(x) y(x) \, dx =
  \int_\Omega [\overline{y}^\ast(x) + \overline{y}_0(x)] y(x) \, dx =
  \int_\Omega \overline{y}^\ast(x)y(x)\, dx.   
\end{equation*}
In the gradient equation \eqref{Eqn:Gradient Equation}, we can thus replace
$\overline{y}$ by its harmonic part $\overline{y}^\ast$. This shows, that we
will always approximate only the harmonic part $\overline{y}^\ast$ of the
target $\overline{y}$.  

In order to include the constraint in the definition of the state space
$Y$, we now consider a Lagrange multiplier $p \in X := H^1_0(\Omega)$, 
and we define the Lagrange functional for $(y,p) \in V \times X$,
$V=H^1(\Omega)$, $\| y \|_V := \| y \|_Y$,
\[
  {\mathcal{L}}(y,p) :=
  \frac{1}{2} \int_\Omega [y(x)-\overline{y}(x)]^2 \, dx +
  \frac{1}{2} \, \varrho \int_\Omega |\nabla y(x)|^2 \, dx +
  \int_\Omega \nabla y(x) \cdot \nabla p(x) \, dx,
\]
where the saddle point $(y_\varrho, p_\varrho) \in V \times X$
is the unique solution of the variational formulation
\begin{eqnarray}
    \int_\Omega y_\varrho(x) \, y(x) \, dx + \varrho
    \int_\Omega \nabla y_\varrho(x) \cdot \nabla y(x) \, dx +
    \int_\Omega \nabla y(x) \cdot \nabla p_\varrho(x) \, dx
    \hspace*{1cm} && \label{Eqn:VF 1} \\[3mm]
    \nonumber
    = \, \int_\Omega \overline{y}(x) \, y(x) \, dx, && \\
  \int_\Omega \nabla y_\varrho(x) \cdot \nabla q(x) \, dx \, = \, 0
  \hspace*{22mm} && \label{Eqn:VF 2}
\end{eqnarray}
for all $(y,q) \in V \times X$. Note, that \eqref{Eqn:VF 1}-\eqref{Eqn:VF 2}
is equivalent to \eqref{Eqn:Gradient Equation}, and hence admits a unique
solution $(y_\varrho,p_\varrho) \in V \times X$.
In addition to the bilinear form
$a(\cdot,\cdot)$, see \eqref{Eqn:Gradient Equation}, we define the
bilinear form
\[
  b(y,p) := \int_\Omega \nabla y(x) \cdot \nabla p(x) \, dx \quad
  \mbox{for} \; (y,p) \in V \times X,
\]
satisfying
\[
  |b(y,p)|
  = \left| \int_\Omega \nabla y(x) \cdot \nabla p(x) \, dx \right| 
  \leq \|  \nabla y \|_{L^2(\Omega)} \| \nabla p \|_{L^2(\Omega)}
  \quad \forall (y,p) \in V \times X .
\]

\begin{remark}\label{Lemma:inf-sup}
  Using the Poincar\'{e} inequality
  $\| p \|_{L^2(\Omega)} \leq c_P \| \nabla p \|_{L^2(\Omega)}$,
  $p\in H^1_0(\Omega)$, we further have
  \begin{eqnarray*}
    \sup_{0\neq y\in H^1(\Omega)} \frac{b(y,p)}{\| y \|_V}
    & = & \sup_{0\neq y\in H^1(\Omega)}
          \frac{b(y,p)}{\sqrt{\varrho \, \| \nabla y \|_{L^2(\Omega)}^2 +
          \| y \|_{L^2(\Omega)}^2}} \\
    & \geq & \frac{b(p,p)}{\sqrt{\varrho \, \| \nabla p \|_{L^2(\Omega)}^2 +
             \| p \|_{L^2(\Omega)}^2}} \\
    & \geq & \frac{\| \nabla p \|_{L^2(\Omega)}^2}
             {\sqrt{(\varrho+c_P^2) \, \| \nabla p \|_{L^2(\Omega)}^2}}
             = \frac{1}{\sqrt{\varrho+c_P^2}} \, \| \nabla p \|_{L^2(\Omega)}.  
  \end{eqnarray*}
  We have therefore derived the inf-sup condition
  \begin{equation}\label{Eqn:inf-sup B}
    c_S \, \| p \|_X \leq \sup_{0\neq y\in V}
    \frac{b(y,p)}{\| y \|_V} 
    \quad \text{for all } p\in X,
  \end{equation}
  with $c_S = (\varrho+c_P)^{-1/2}>0$, which also implies unique solvability
  of the system \eqref{Eqn:VF 1}-\eqref{Eqn:VF 2} by Brezzi's theorem
  \cite{Brezzi1974RAIRO}. In particular, this will be important for the
  well-posedness of the discrete problem.  
\end{remark}

\noindent
In Remark \ref{Lemma:inf-sup}, we have used the norm
estimate $\| p \|_V \leq (\varrho+c_P^2)^{1/2} \, \| p \|_X$ for all
$p \in X$. For $y \in V$, the opposite estimate is also valid:

\begin{lemma}
  For any $y \in V$, there holds
  \begin{equation}\label{Eqn:Norm X Y}
    \| y \|_X \, \leq \, \frac{1}{\sqrt{\varrho}} \, \| y \|_V \, .
  \end{equation}
\end{lemma}

\begin{proof}
  For any $y \in Y$ we consider
  \[
    \| y \|_X^2 = \| \nabla y \|^2_{L^2(\Omega)} \, \leq \,
    \frac{1}{\varrho} \, \left( \| y \|_{L^2(\Omega)}^2 +
      \varrho \, \| \nabla y \|^2_{L^2(\Omega)} \right) \, = \,
    \frac{1}{\varrho} \, \| y \|_V^2 
  \]
  to conclude the assertion.
\end{proof}

\noindent
When considering the variational formulation \eqref{Eqn:VF 1}
for $y = q \in H^1_0(\Omega)$, this gives, using \eqref{Eqn:VF 2},
\[
  \int_\Omega \nabla p_\varrho(x) \cdot \nabla q(x) \, dx \\
  =
  \int_\Omega [\overline{y}(x) - y_\varrho(x)] \, q(x) \, dx,
\]
i.e., $p_\varrho \in H^1_0(\Omega)$ is the weak solution of the
Dirichlet boundary value problem
\begin{equation}\label{Eqn:Adjoint problem}
  - \Delta p_\varrho = \overline{y} - y_\varrho \quad \mbox{in} \; \Omega,
  \quad p_\varrho = 0 \quad \mbox{on} \; \Gamma .
\end{equation}
Now, considering \eqref{Eqn:VF 1} for an arbitrary
$y \in H^1(\Omega)\backslash H^1_0(\Omega)$, this gives
\begin{eqnarray*}
  \int_\Omega [\overline{y}(x) - y_\varrho(x)] \, y(x) \, dx
  & = & \int_\Omega \nabla y(x) \cdot \nabla p_\varrho(x) \, dx +
        \varrho \int_\Omega \nabla y_\varrho(x) \cdot \nabla y(x) \, dx \\
  && \hspace*{-4cm} = \, \int_\Gamma
     \frac{\partial}{\partial n_x} p_\varrho(x) \, y(x) \, ds_x +
     \int_\Omega [- \Delta p_\varrho(x)] \, y(x) \, dx +
     \varrho \int_\Gamma \frac{\partial}{\partial n_x}
     y_\varrho(x) \, y(x) \, ds_x 
\end{eqnarray*}
i.e.,
\begin{equation}\label{Eqn:gradient equation}
  \int_\Gamma
  \frac{\partial}{\partial n_x} p_\varrho(x) \, y(x) \, ds_x
  +
  \varrho \int_\Gamma \frac{\partial}{\partial n_x}
  y_\varrho(x) \, y(x) \, ds_x = 0 \quad
  \forall
  \; y \in H^1(\Omega \backslash H^1_0(\Omega)) .
\end{equation}
The optimality system consisting of the primal Dirichlet boundary value
problem \eqref{Eqn:Laplace DBVP}, the adjoint Dirichlet problem
\eqref{Eqn:Adjoint problem}, and \eqref{Eqn:gradient equation} was already
considered in \cite{OfPhanSteinbach2015NumerMath},
where a different approach in the mathematical and
numerical analysis was applied.

%
%

\section{Finite element discretization}\label{Sec:FED}
Recall the variational formulations \eqref{Eqn:VF 1} and \eqref{Eqn:VF 2}
to find $(y_\varrho,p_\varrho) \in V \times X$ such that
\begin{equation}\label{Eqn:VF 12}
  \begin{array}{lclcl}
    a(y_\varrho,y) + b(y,p_\varrho)
    & = & \langle \overline{y},y \rangle_{L^2(\Omega)}
    && \forall y \in V, \\[1mm]
    b(y_\varrho,q) & = & 0 & & \forall q \in X,
  \end{array}
\end{equation}
which is now discretized by means of the finite element method.
Let ${\mathcal{T}}_h = \{ \tau_\ell \}_{\ell=1}^n$ be a family of admissible
decompositions of $\Omega$ into shape regular simplicial finite elements
of mesh size $h$ which are assumed to be globally quasi-uniform;
$V_h := S_h^1(\Omega) = \mbox{span} \{ \varphi_k \}_{k=1}^M \subset H^1(\Omega)$
is the related space of piecewise linear and continuous basis functions
$\varphi_k$, and $X_h := S_h^1(\Omega) \cap H^1_0(\Omega) =
\mbox{span} \{ \varphi_k \}_{k=1}^N \subset H_0^1(\Omega)$,
where $M=M_h > N=N_h$.
We note that the basis functions $\varphi_{N+1},\ldots,\varphi_M$ belong to
the nodes located on the boundary $\partial \Omega$.
The Galerkin finite element discretization of \eqref{Eqn:VF 12}
is to find $(y_{\varrho h},p_{\varrho h}) \in V_h \times X_h$ such that
\begin{equation}\label{Eqn:FEM}
  \begin{array}{lclcl}
    a(y_{\varrho h},y_h) + b(y_h,p_{\varrho h})
    & = & \langle \overline{y},y_h \rangle_{L^2(\Omega)}
    && \forall y_h \in V_h, \\[1mm]
    b(y_{\varrho h},q_h) & = & 0 & & \forall q_h \in X_h.
  \end{array}
\end{equation}
We introduce the space 
\begin{equation*}
 Y_h = \Big \{ y_h \in V_h:\, b(y_h,q_h)=0,\, \forall q_h\in X_h\Big\}
\end{equation*}
of discrete harmonic functions.
Then, \eqref{Eqn:FEM} is equivalent to the variational formulation to
find $y_{\varrho h}\in Y_h$ such that  
\begin{equation}\label{Eqn:FEM equivalent}
  a(y_{\varrho h},y_h) =
  \langle \overline{y},y_h \rangle_{L^2(\Omega)}
  \quad \text{for all } y_h\in Y_h. 
\end{equation}
We note that, in this case, we cannot deduce unique solvability of
\eqref{Eqn:FEM equivalent}, as $Y_h \not \subset Y$ and we are hence
dealing with a non-conforming discretization. But the inf-sup stability
condition \eqref{Eqn:inf-sup B} remains true for all $p_h \in X_h$ when the
supremum is taken over all $y_h \in V_h$. Hence, by
Remark \ref{Lemma:inf-sup}, we conclude unique solvability of \eqref{Eqn:FEM}.
To derive a priori error estimates, we first follow the standard
approach in mixed finite element methods, e.g.,
\cite[Theorem 7.4.3]{QuarteroniValli1994Book}.

\begin{lemma}
  Let $(y_\varrho,p_\varrho) \in V \times X$ and
  $(y_{\varrho h},p_{\varrho h}) \in V_h \times X_h$ be the unique
  solutions of the variational formulations \eqref{Eqn:VF 12} and
  \eqref{Eqn:FEM}, respectively. Then 
  the error estimate
  \begin{equation}\label{Eqn:Cea}
    \| y_\varrho -y_{\varrho h}\|_V
    \, \leq \, 2 \,
    \| y_\varrho-y_h\|_V  +
    \frac{1}{\sqrt{\varrho}} \, \| p_\varrho-q_h \|_X
  \end{equation}
  holds for all $(y_h,q_h) \in Y_h \times X_h$.
\end{lemma}

\begin{proof}
  When subtracting the first equation in \eqref{Eqn:FEM} from the first
  equation in \eqref{Eqn:VF 12} for $y=y_h \in V_h \subset V$, this gives
  \[
    a(y_\varrho-y_{\varrho h},y_h) + b(y_h,p_\varrho-p_{\varrho h}) \, = \, 0
    \quad \forall y_h \in V_h .
  \]
   This is equivalent to
   \[
     a(y_{\varrho h}-y_h^*,y_h) + b(y_h,p_{\varrho h}-q_h) 
     \, = \,
     a(y_\varrho-y_h^*,y_h) + b(y_h,p_\varrho-q_h) \quad \forall y_h \in V_h,
   \]
   where $(y_h^*,q_h) \in Y_h \times X_h$ are arbitrary functions, i.e.,
   \[
     b(y_h^*,q_h) = 0 \quad \forall q_h \in X_h .
   \]
   For the particular test function $y_h = y_{\varrho h} - y_h^*$ we have
   $b(y_{\varrho h}-y_h^*,p_{\varrho h}-q_h) = 0$, and hence we conclude,
   using \eqref{Eqn:Norm X Y},
   \begin{eqnarray*}
     \| y_{\varrho h} - y_h^* \|_V^2
     & = & a(y_{\varrho h} - y_h^*,y_{\varrho h} - y_h^*) \\
     & = & a(y_\varrho-y_h^*,y_{\varrho h}-y_h^*) +
           b(y_{\varrho h}-y_h^*,p_\varrho-q_h) \\
     & = & \| y_\varrho-y_h^*\|_V \| y_{\varrho h}-y_h^*\|_V +
           \| y_{\varrho h}-y_h^* \|_X \| p_\varrho-q_h \|_X \\
     & = & \left[ \| y_\varrho-y_h^*\|_V  +
           \frac{1}{\sqrt{\varrho}} \, \| p_\varrho-q_h \|_X
           \right] \, \| y_{\varrho h}-y_h^*\|_V,
   \end{eqnarray*}
  and the assertion follows from the triangle inequality.
\end{proof}

\begin{lemma}
  Assume that $\Omega$ is either a convex Lipschitz domain, or a bounded
  domain with smooth boundary.
  Let $(y_\varrho,p_\varrho) \in V \times X$ be the unique solution of the
  variational formulation \eqref{Eqn:VF 12}. Then $p_\varrho\in H^2(\Omega)$ and
  \begin{equation}\label{Eqn:Error p}
    \inf_{q_h\in X_h}\| p_\varrho - q_h \|_X \leq \norm{p_\varrho-I_h p_\varrho}_X\leq 
    c \, h \, \| y_\varrho - \overline{y} \|_{L^2(\Omega)} .
  \end{equation}
\end{lemma}

\begin{proof}
  For the solution $p_\varrho \in H^1_0(\Omega)$ of the adjoint problem
  \eqref{Eqn:Adjoint problem} for
  $\overline{y}-y_\varrho \in L^2(\Omega)$ we conclude
  $p_\varrho \in H^2(\Omega)$, due to the assumptions made on $\Omega$. 
  Using standard interpolation error estimates, 
  we get
  \[
    \| p_\varrho - I_h p_\varrho \|_X =
    \| \nabla ( p_\varrho - I_h p_\varrho) \|_{L^2(\Omega)} \leq
    c \, h \, |p_\varrho|_{H^2(\Omega)} \leq
    c \, h \, \| y_\varrho - \overline{y} \|_{L^2(\Omega)} .
  \]
\end{proof}

\begin{remark}
  The assumption that $\Omega$ is convex or smoothly bounded is needed, in
  order to guarantee that $p_\varrho\in H^2(\Omega)$ for all
  $\overline{y} \in L^2(\Omega)$. If $\Omega$ is an arbitrary Lipschitz domain we can at least expect $p_\varrho\in H^{3/2}(\Omega)$, see \cite[Cor. 3.7 (ii)]{BehrndtGesztesyMitrea2025}. 
\end{remark}

\begin{lemma}
  Let $(y_\varrho,p_\varrho) \in V \times X$ be the unique solution of the
  variational formulation \eqref{Eqn:VF 12}, and assume
  $\overline{y} \in Y \cap H^2(\Omega)$. Then
  there holds the error estimate
  \begin{equation}\label{Eqn:Error yrho}
    \inf\limits_{y_h^* \in Y_h}
    \| y_\varrho - y_h^* \|_V \leq c \, 
    \sqrt{h^4 + \varrho h^2 + \varrho^2} \, \| \overline{y} \|_{H^2(\Omega)} =
    \widetilde{c} \, h^2 \, \| \overline{y} \|_{H^2(\Omega)},
  \end{equation}
  when choosing $\varrho = h^2$ in the case of a smooth boundary, and 
  \begin{eqnarray}
    \inf\limits_{y_h^* \in Y_h}
    \| y_\varrho - y_h^* \|_V
    &\leq & c \, \sqrt{h^4 + \varrho h^2 + \varrho^2(1+|\log(\varrho)|)^2} \,
            \| \overline{y} \|_{H^2(\Omega)} \nonumber \\
    & = & \label{Eqn:Error yrho p.w. 1}
    \widetilde{c} \, h^2 \, \| \overline{y} \|_{H^2(\Omega)},
  \end{eqnarray}
  choosing $\varrho = h^2/|\log(h)|$ in the case of a domain with piecewise
  smooth boundary.
\end{lemma}

\begin{proof}
  When using the triangle inequality, the results of
  Lemma \ref{Lemma:Regularization error smooth}, and standard error estimates
  for the finite element solution $y_h^* \in V_h$ of
  $- \Delta \overline{y}=0$ in $\Omega$, this gives
  \begin{eqnarray*}
    \| y_\varrho - y_h^* \|_V^2
    & \leq & 2 \, \| y_\varrho - \overline{y} \|_V^2 +
             2 \, \| \overline{y} - y_h^* \|_V^2 \\
    & = & 2 \Big[ \| y_\varrho - \overline{y} \|_{L^2(\Omega)}^2 +
          \varrho \, \| \nabla (y_\varrho - \overline{y}) \|^2_{L^2(\Omega)}
          \Big] \\
    && \hspace*{3cm} +
          2 \Big[ \| \overline{y} - y_h^* \|_{L^2(\Omega)}^2 +
          \varrho \, \| \nabla (\overline{y}-y_h^*) \|^2_{L^2(\Omega)}
       \Big] \\
    & \leq & c_1 \, \varrho^2 \, \| \overline{y} \|_{H^2(\Omega)}^2
             + c_2 \,
             \Big[ h^4 + \varrho \, h^2 \Big] \, |\overline{y}|_{H^2(\Omega)}^2 .
  \end{eqnarray*}
  The estimate \eqref{Eqn:Error yrho p.w. 1} is derived in the same way
  using the results of Lemma \ref{Lemma:Regularization error} which give
  \begin{equation*}
    \| y_{\varrho h}-\overline{y} \|_{L^2(\Omega)}^2 \leq
    c \, \Big[ h^4 + \varrho \, h^2 + \varrho^2 \, (1+|\log(\varrho)|)^2
    \Big] \, \| \overline{y} \|_{H^2(\Omega)}^2. 
  \end{equation*}
  Now we compute 
  \begin{eqnarray*}
    h^4 + \varrho \, h^2 + \varrho^2 \, (1+|\log(\varrho)|)^2
    & = & h^4 + \varrho \, h^2 + \varrho^2 + 2 \, \varrho^2 \,
          |\log(\varrho)| + (\varrho \, |\log(\varrho)|)^2 \\
    & = & h^4 + \varrho \, h^2 + \varrho^2 \, (1+ 2 \, |\log(\varrho)| +
          |\log(\varrho)|^2) .
  \end{eqnarray*}
  If $\varrho = h^2/|\log(h)|$  and $h$ is small, we get
  \[
    1 + 2\, |\log(\varrho)| + |\log(\varrho)|^2 \simeq
    |\log(\varrho)|^2 \simeq 4 \, |\log(h)|^2,
  \]
  and  
  \begin{eqnarray*}
    h^4 + \varrho \, h^2 + \varrho^2 \,
    (1+ 2 \, \log(\varrho) + |\log(\varrho)|^2)
    & \simeq & h^4 + \frac{h^4}{|\log(h)|} +
               \frac{h^4}{|\log(h)|^2} \, 4 \, |\log(h)|^2 \\
    & \simeq & h^4.
  \end{eqnarray*}
\end{proof}

\noindent
Now we are in the position to formulate the main result of this section:

\begin{theorem}\label{Theorem:DiscretizationErrorEstimate} 
  Let $y_{\varrho h} \in V_h$ be the unique solution of the variational
  formulation \eqref{Eqn:FEM}. Assume that $\Omega$ is a bounded domain
  with a smooth boundary, and assume $\overline{y} \in Y \cap H^2(\Omega)$.
  Then there holds the error estimate
  \begin{equation}\label{Eqn:Error regular}
    \| y_{\varrho h} - \overline{y} \|_{L^2(\Omega)} \leq
    c \, \sqrt{h^4 + \varrho h^2 + \varrho^2} \,
    \| \overline{y} \|_{H^2(\Omega)} =
    \widetilde{c} \, h^2 \, \| \overline{y} \|_{H^2(\Omega)},
  \end{equation}
  when choosing $\varrho = h^2$. If the boundary is piecewise smooth, we get 
  \begin{equation*}
    \| y_{\varrho h} - \overline{y} \|_{L^2(\Omega)} \leq
    c \, \sqrt{h^4 + \varrho h^2 + \varrho^2(1+|\log(\varrho)|)^2} \,
    \| \overline{y} \|_{H^2(\Omega)} \leq \widetilde{c} \, h^2 \,
    \|\overline{y}\|_{H^2(\Omega)},
  \end{equation*}
  when choosing $\varrho = h^2/|\log(h)|$.
\end{theorem}

\begin{proof}
  When using the simple estimate $(a+b)^2 \leq 2(a^2+b^2)$,
  Cea's lemma \eqref{Eqn:Cea} for $q_h = I_h p_\varrho$ and
  $y_h^* \in V_h$ as finite element approximation of $\overline{y}$,
  the regularization error estimate
  \eqref{Eqn:Regularization error regular smooth},
  the error estimates \eqref{Eqn:Error yrho} and \eqref{Eqn:Error p},
  this gives
  \begin{eqnarray*}
    \| y_{\varrho h} - \overline{y} \|^2_{L^2(\Omega)}
    & \leq & 2 \, \| y_{\varrho h} - y_\varrho \|_{L^2(\Omega)}^2 +
             2 \, \| y_\varrho - \overline{y} \|_{L^2(\Omega)}^2 \\
    && \hspace*{-2.5cm} \leq \, 2 \, \left( 2 \, \| y_\varrho - y_h^* \|_V +
             \frac{1}{\sqrt{\varrho}} \,
             \| p_\varrho - I_h p_\varrho \|_X \right)^2
             + c_1 \, \varrho^2 \, \| \overline{y} \|^2_{H^2(\Omega)} \\
    && \hspace*{-2.5cm}
       \leq \, 8 \, \| y_\varrho - y_h^* \|^2_V +
             \frac{4}{\varrho} \,
             \| p_\varrho - I_h p_\varrho \|_X^2 
             + c_1 \, \varrho^2 \, \| \overline{y} \|^2_{H^2(\Omega)} \\
    && \hspace*{-2.5cm}
       \leq \, c_2 \Big[ h^4 + \varrho h^2 + \varrho^2 \Big]
             \, \| \overline{y} \|^2_{H^2(\Omega)} +
             \frac{4}{\varrho} \, c_3 \, h^2 \,
             \underbrace{\| y_\varrho - \overline{y} \|_{L^2(\Omega)}^2
             }_{\leq c_4 \varrho^2 \| \overline{y} \|_{H^2(\Omega)}^2}
             + c_1 \, \varrho^2 \,
             \| \overline{y} \|^2_{H^2(\Omega)} \\[-2mm]
    && \hspace*{-2.5cm}
       \leq \, c \, \Big[ h^4 + \varrho h^2 + \varrho^2 \Big] \,
             \| \overline{y} \|^2_{H^2(\Omega)},
  \end{eqnarray*}
  i.e., the assertion. In the case of a piecewise smooth boundary we follow
  the same lines but replace \eqref{Eqn:Regularization error regular smooth}
  by \eqref{Eqn:Regularization error pw 1} and \eqref{Eqn:Error yrho} by
  \eqref{Eqn:Error yrho p.w. 1}.
\end{proof}

\noindent
Next we consider the situation when the target $\overline{y}$ is less regular:

\begin{theorem}
  Let $y_{\varrho h} \in V_h$ be the unique solution of the variational
  formulation \eqref{Eqn:FEM}. Assume that $\Omega$ is either a convex Lipschitz domain, or a bounded domain with smooth boundary. For $\overline{y} \in Y$ there holds the error estimate
  \begin{equation}\label{Eqn:Error Y}
    \| y_{\varrho h} - \overline{y} \|_{L^2(\Omega)} \leq
    c \, \sqrt{h^2 + \varrho} \, \| \nabla \overline{y} \|_{L^2(\Omega)} =
    \widetilde{c} \, h \, \| \nabla \overline{y} \|_{L^2(\Omega)},
  \end{equation}
  when choosing $\varrho = h^2$, while for $\overline{y} \in L^2(\Omega)$
  we have
   \begin{equation}\label{Eqn:Error L2}
    \| y_{\varrho h} - \overline{y} \|_{L^2(\Omega)} \leq
    c \, \sqrt{1 + h^2 \varrho^{-1}} \, \| \overline{y} \|_{L^2(\Omega)} =
    \widetilde{c} \, \| \overline{y} \|_{L^2(\Omega)}.
  \end{equation}
\end{theorem}

\begin{proof}
  For $\overline{y} \in Y$ and using \eqref{Eqn:Error p} and
  \eqref{Eqn:Regularization error Y} we first have
  \[
    \| p_\varrho - I_h p_\varrho \|_X \leq
    c \, h \, \| y_\varrho - \overline{y} \|_{L^2(\Omega)} \leq
    c \, h \, \varrho^{1/2} \, \| \nabla \overline{y} \|_{L^2(\Omega)} .
  \]
  For $y_h^* \in V_h$ being the finite element approximation of
  $y_\varrho \in Y$ we now have, using standard finite element error
  estimates and \eqref{Eqn:Regularization error Y},
  \begin{eqnarray*}
    \| y_\varrho - y_h^* \|_V^2
    & = & \| y_\varrho - y_h^* \|_{L^2(\Omega)}^2 + \varrho \,
          \| \nabla (y_\varrho - y_h^*) \|^2_{L^2(\Omega)} \\
    & \leq & c \, \Big[ h^2 + \varrho \Big] \,
             \| \nabla y_\varrho \|^2_{L^2(\Omega)} \, \leq \,
             c \, \Big[ h^2 + \varrho \Big] \,
             \| \nabla \overline{y} \|^2_{L^2(\Omega)},
  \end{eqnarray*}
  Cea's lemma \eqref{Eqn:Cea} now gives
  \[
    \| y_\varrho - y_{\varrho h} \|^2_{L^2(\Omega)}
    \leq 4 \, \| y_\varrho - y_h^* \|_V^2 + \frac{2}{\varrho} \,
    \| p_\varrho - I_h p_\varrho \|_X^2 \leq
    c \, \Big[ h^2 + \varrho \Big] \,
    \| \nabla \overline{y} \|^2_{L^2(\Omega)} .
  \]
  Hence, and using \eqref{Eqn:Regularization error Y} we
  obtain \eqref{Eqn:Error Y},
  \[
    \| y_{\varrho h} - \overline{y} \|_{L^2(\Omega)}^2
    \leq 2 \, \| y_{\varrho h} - y_\varrho \|^2_{L(\Omega)} +
    2 \, \| y_\varrho - \overline{y} \|^2_{L^2(\Omega)} \\
    \leq c \, \Big[ h^2 + \varrho \Big] \,
    \| \nabla \overline{y} \|^2_{L^2(\Omega)} .
  \]
  For $\overline{y} \in L^2(\Omega)$ we conclude,
  now using \eqref{Eqn:Regularization error L2},
   \[
    \| p_\varrho - I_h p_\varrho \|_X \leq
    c \, h \, \| y_\varrho - \overline{y} \|_{L^2(\Omega)} \leq
    c \, h \, \| \overline{y} \|_{L^2(\Omega)} ,
  \]
  and
  \[
    \| y_\varrho - y_h^* \|_V^2 \leq c \, \Big[ h^2 + \varrho \Big] \,
    \| \nabla y_\varrho \|^2_{L^2(\Omega)}
    \leq c \, \Big[ h^2 + \varrho \Big] \, \varrho^{-1} \,
    \| \overline{y} \|^2_{L^2(\Omega)} .
  \]
  With this, \eqref{Eqn:Error L2} follows.
\end{proof}

%
%

\section{Solver}
\label{Sec:Solver}
Once the basis is introduced, the mixed finite element scheme \eqref{Eqn:FEM}
can be rewritten as the following symmetric, but indefinite linear system of
algebraic equations: Find  $\mathbf{y}_h \in \mathbb{R}^M$ and
$\mathbf{p}_h \in \mathbb{R}^N$ such that
\begin{equation}\label{Eqn:SID-System}
  \left(
    \begin{array}{cc}
      \mathbf{M}_h + \varrho \mathbf{K}_h & \widetilde{\mathbf{K}}_h^\top \\[1mm]
      \widetilde{\mathbf{K}}_h & \mathbf{0}_h
    \end{array}
  \right)
  \left(
    \begin{array}{c}
      \mathbf{y}_h \\[1mm]
      \mathbf{p}_h
    \end{array}
  \right)
  =
  \left(
    \begin{array}{c}
      \overline{\mathbf{y}}_h \\[1mm]
      \mathbf{0}_h
    \end{array}
  \right),
\end{equation}
where $\mathbf{M}_h = (\langle \varphi_j,\varphi_i \rangle_{L^2(\Omega)}
)_{i,j=1,\ldots,M}$ denotes the $M \times M$ symmetric and positive definite
mass matrix, $\mathbf{K}_h = (\langle \nabla \varphi_j,  \nabla \varphi_i
\rangle_{L^2(\Omega)})_{i,j=1,\ldots,M}$ is the $M \times M$ symmetric, but
singular Neumann stiffness matrix, $\widetilde{\mathbf{K}}_h =
(\langle \nabla\varphi_j,  \nabla \varphi_i \rangle_{L^2(\Omega)}
)_{i=1,\ldots,N,j=1,\ldots,M}$ is a rectangular matrix of the dimension
$N \times M$, and $\overline{\mathbf{y}}_h =
(\langle \overline{y},\varphi_i \rangle_{L^2(\Omega)})_{i=1,\ldots,M} \in
\mathbb{R}^M$ is computed from the given target $\overline{y}$. The solution
vectors $\mathbf{y}_h$ and $\mathbf{p}_h$ of \eqref{Eqn:SID-System} 
are related to the solutions $y_{\varrho h} \in V_h$ and $p_{\varrho h}  \in X_h$ 
of the mixed finite element scheme \eqref{Eqn:FEM} by the finite element
isomorphism written as $\mathbf{y}_h \leftrightarrow y_{\varrho h} $ and
$\mathbf{p}_h \leftrightarrow p_{\varrho h}$. When eliminating 
\begin{equation}\label{Eqn:Elimination}
  \mathbf{y}_h = (\mathbf{M}_h + \varrho \mathbf{K}_h)^{-1}
  (\overline{\mathbf{y}}_h - \widetilde{\mathbf{K}}_h^\top \mathbf{p}_h)
\end{equation}
from \eqref{Eqn:SID-System}, we arrive at the dual Schur complement system 
\begin{equation}\label{Eqn:DualSchurComplementSystem}
  \widetilde{\mathbf{K}}_h [\mathbf{M}_h+\varrho \mathbf{K}_h]^{-1}
  \widetilde{\mathbf{K}}_h^\top \mathbf{p}_h =
  \widetilde{\mathbf{K}}_h [\mathbf{M}_h +
  \varrho \mathbf{K}_h]^{-1} \overline{\mathbf{y}}_h
\end{equation}
for defining $\mathbf{p}_h \in \mathbb{R}^N$. Once $\mathbf{p}_h$ is
determined, we can easily compute the optimal finite element state
$\mathbf{y}_h \leftrightarrow y_{\varrho h}$ via \eqref{Eqn:Elimination}, and
the corresponding optimal control $u_{\varrho h} = \gamma_0^{int} y_{\varrho h}$ 
as Dirichlet trace of $y_{\varrho h}$ on $\Gamma$.

First of all, we observe that the matrix
$[\mathbf{M}_h+\varrho \mathbf{K}_h]^{-1}$ in the dual Schur complement is
spectrally equivalent to $[\mathbf{M}_h]^{-1}$ for the choice $\varrho = h^2$
that has proved to be optimal for the discretization error estimates 
in the case of smooth boundaries. This obviously remains true for
$\varrho \le h^2$; i.e., for the choice $\varrho = h^2/|\log(h)|$ made in
Theorem~\ref{Theorem:DiscretizationErrorEstimate} for piecewise smooth
boundaries. More precisely, the spectral equivalence inequalities
\begin{equation}\label{Eqn:SpectralEquivalenceInequalities1}
  \underline{c}_M \, h^d \, \mathbf{I}_h \le
  \frac{1}{d+2} \, \mathbf{D}_h \le \mathbf{M}_h \le
  \mathbf{M}_h+\varrho \mathbf{K}_h 
  \le c \, \mathbf{M}_h \le c \, \mathbf{D}_h \le
  \overline{c}_M \, h^d \, \mathbf{I}_h
\end{equation}
hold for $\varrho \le h^2$, where $\mathbf{I}_h$ denotes the identity matrix,
and $\mathbf{D}_h = \mbox{lump}(\mathbf{M}_h)$ is the lumped mass matrix.
In \eqref{Eqn:SpectralEquivalenceInequalities1}, 
$c = 1 + c_{inv}^2 \ge  1 + \varrho h^{-2} c_{inv}^2$, and $c_{inv}$ is the
constant in the inverse inequality 
$\|v_h\|_{H^1(\Omega)} \le c_{inv} \, h^{-1} \, \|v_h\|_{L^2(\Omega)}$
for all $v_h \in V_h$, $\underline{c}_M$ and $\overline{c}_M$ are positive
constants which are independent on $h$. We refer to
\cite{LangerLoescherSteinbachYang:2024NLA} for a derivation of inequalities
like \eqref{Eqn:SpectralEquivalenceInequalities1}.
From \eqref{Eqn:SpectralEquivalenceInequalities1}, we deduce that 
the dual Schur complement $\widetilde{\mathbf{K}}_h
[\mathbf{M}_h+\varrho \mathbf{K}_h]^{-1} \widetilde{\mathbf{K}}_h^\top$
is spectrally equivalent to 
$\widetilde{\mathbf{K}}_h \mathbf{M}_h^{-1} \widetilde{\mathbf{K}}_h^\top$.
The latter matrix is nothing but the Schur complement matrix arising
from the mixed finite element discretization of the first biharmonic 
boundary value problem that was originally proposed by Ciarlet and
Raviart \cite{CiarleRaviart1974Proceedings}.

Several preconditioners were proposed for the symmetric and positive
definite matrix
$\widetilde{\mathbf{K}}_h \mathbf{M}_h^{-1} \widetilde{\mathbf{K}}_h^\top$
in the literature. In our our numerical experiments, we use the preconditioner 
$\mathbf{C}_h = \mathbf{K}_{0h}^2$, where $\mathbf{K}_{0h} =
(\langle \nabla\varphi_j,  \nabla \varphi_i
\rangle_{L^2(\Omega)})_{i,j=1,\ldots,N}$ is nothing but the Dirichlet stiffness
matrix for the Laplacian. This preconditioner was analyzed in
\cite{BraessPeisker:1986IMAJNA,Langer1986NumerMath}. More precisely, it
was shown that there exist positive, $h$-independent constants
$\underline{c}_0$ and $\overline{c}_0$ such that
\begin{equation}\label{Eqn:SpectralEquivalenceInequalities2} 
  \underline{c}_0 \, h^{-d} \, \mathbf{C}_h
  \le \widetilde{\mathbf{K}}_h \mathbf{M}_h^{-1} \widetilde{\mathbf{K}}_h^\top
  \le \overline{c}_0 h^{-d-1} \, \mathbf{C}_h. 
\end{equation}
The same inequalities with modified constants $\underline{c}_0$ and
$\overline{c}_0$ hold if
$\widetilde{\mathbf{K}}_h \mathbf{M}_h^{-1} \widetilde{\mathbf{K}}_h^\top$
is replaced by the dual Schur complement 
$\widetilde{\mathbf{K}}_h [\mathbf{M}_h+\varrho \mathbf{K}_h]^{-1}
\widetilde{\mathbf{K}}_h^\top$ in which we are interested.

It is clear that we can use the scaled versions 
\begin{equation}\label{Eqn:Preconditioners} 
  \mathbf{C}_h = h^{-d} \, \mathbf{K}_{0h}^2 \quad \mbox{or} \quad
  \mathbf{C}_h = \mathbf{K}_{0h} [\mbox{lump}(\mathbf{M}_{0h})]^{-1}
  \mathbf{K}_{0h}  
\end{equation}
as preconditioners for $\widetilde{\mathbf{K}}_h
[\mathbf{M}_h+\varrho \mathbf{K}_h]^{-1} \widetilde{\mathbf{K}}_h^\top$,
where $\mathbf{M}_{0h} =
(\langle \varphi_j,  \varphi_i \rangle_{L^2(\Omega)})_{i,j=1,\ldots,N}$ 
is the mass matrix which is built from the interior basis functions.
More precisely, there are 
positive, $h$-independent constants
$\underline{c}_1$ and $\overline{c}_1$ such that
\begin{equation}\label{Eqn:SpectralEquivalenceInequalities3} 
  \underline{c}_1  \, \mathbf{C}_h \le
  \widetilde{\mathbf{K}}_h [\mathbf{M}_h+\varrho \mathbf{K}_h]^{-1}
  \widetilde{\mathbf{K}}_h^\top \le
  \overline{c}_1 \, h^{-1} \, \mathbf{C}_h. 
\end{equation}
These spectral inequalities immediately yield that we can solve 
the dual Schur-complement system \eqref{Eqn:DualSchurComplementSystem} 
by means of the PCG with one of the preconditioners $\mathbf{C}_h$
within ${\mathcal{O}}(h^{-1/2} \ln(\varepsilon^{-1}))$ iterations,
where $\varepsilon \in (0,1)$ denotes the prescribed relative accuracy.
In the preconditioning step of every PCG iteration, we have to solve two
systems with the system matrix $\mathbf{K}_{0h}$. In order to perform this
in optimal complexity, special multigrid methods can be used as discussed
in \cite{BraessPeisker:1986IMAJNA}. Furthermore, the application of the 
$\widetilde{\mathbf{K}}_h [\mathbf{M}_h+\varrho \mathbf{K}_h]^{-1}
\widetilde{\mathbf{K}}_h^\top$ to a vector as requested in every PCG
iteration step needs the solution of a system with the system matrix
$\mathbf{M}_h+\varrho \mathbf{K}_h$, but this can efficiently be done by
some inner iteration since $\mathbf{M}_h+\varrho \mathbf{K}_h$
is well-conditioned, cf. \eqref{Eqn:SpectralEquivalenceInequalities1}.
The PCG solver proposed is not of optimal complexity due to the fact that 
the number of iteration grows like
${\mathcal{O}}(h^{-1/2} \ln(\varepsilon^{-1}))$. This is a moderate growth
in comparison with the plane CG; cf. also the numerical results provided
in Section~\ref{Sec:NumericalResults}.
This growth can be avoided if we use more 
sophisticated
preconditioning 
procedures that take care on the right boundary behavior; see
\cite{ArioliLoghin2009SINUM,GlowinskiPironneau:1979SIAMRev,
JohnSteinbach2013Bericht4,Peisker1988M2AN}.


\section{Control and state constraints}\label{Sec:Constraints}
When using a state-based approach for the solution of Dirichlet boundary
control problems, we can incorporate state and control constraints at once.
We now consider the minimization of the reduced cost functional 
\eqref{Eqn:reduced functional} over the convex, bounded and non-empty subset
\begin{eqnarray*}
  K & := & \Big \{ y \in Y : g_- \leq u_\varrho = \gamma_0^{int} y_\varrho
           \leq g_+ \; \mbox{almost everywhere (a.e.) on} \; \Gamma \Big 
           \} \\
    & = & \Big \{ y \in H^1(\Omega) :
          \langle \nabla y , \nabla v \rangle_{L^2(\Omega)} = 0 \;\,
          \forall v \in H^1_0(\Omega),\;
          g_- \leq y \leq g_+ \; \mbox{a.e. in} \; \Omega \Big \},
\end{eqnarray*}
where $g_-, g_+ \in \mathbb{R}$ with $g_- \le 0 \le g_+$. Note that the
last assumption is required in the proof of related regularization error
estimates \cite{GanglLoescherSteinbach:2025}, but this condition can easily
be satisfied by simple additive scaling. The equality of the both sets
follows from the maximum principle for harmonic functions. We note that $K$
involves both equality and inequality constraints. Hence we find the
minimizer $y_\varrho \in K$ of \eqref{Eqn:reduced functional} as unique
solution $(y_\varrho,p_\varrho) \in \overline{K} \times H^1_0(\Omega)$
satisfying
\[
  \begin{array}{lcl}
  \langle y_\varrho - \overline{y}, y - y_\varrho \rangle_{L^2(\Omega)} +
  \varrho \, \langle \nabla y_\varrho , \nabla (y- y_\varrho)
  \rangle_{L^2(\Omega)} +
  \langle \nabla p_\varrho , \nabla (y-y_\varrho) \rangle_{L^2(\Omega)}
  & \geq & 0, \\[2mm]
    \langle \nabla y_\varrho , \nabla q \rangle_{L^2(\Omega)} & = & 0
  \end{array}
\]
for all $(y,q) \in \overline{K} \times H^1_0(\Omega)$, where
$ \overline{K} := \{ y \in H^1(\Omega) : g_- \leq y \leq g_+ \;
\mbox{in} \; \Omega \; \mbox{a.e.} \}$.
As in the unconstrained case we use the finite element space
$X_h = S_h^1(\Omega) \cap H^1_0(\Omega)$, but now we define
\[
  \overline{K}_h := \Big \{ y_h =
  \sum\limits_{k=1}^M y_k \varphi_k \in S_h^1(\Omega) :
  g_- \leq y_k \leq g_+ \; \mbox{for all} \; k=1,\ldots,M \Big \}.
\]
The finite element discretization then results in the 
discrete variational inequality to find
$\underline{y}_\varrho \in {\mathbb{R}}^M \leftrightarrow
y_{\varrho h} \in \overline{K}_h$ such that
\[
  (M_h \underline{y}_\varrho + \varrho K_h \underline{y}_\varrho +
  \widetilde{K}_h^\top \underline{p}_\varrho - \underline{f} ,
  \underline{y} - \underline{y}_\varrho) \geq 0 \quad
  \mbox{for all} \; \underline{y} \in {\mathbb{R}}^M
  \leftrightarrow y_h \in \overline{K}_h,
\]
with the constraint $\widetilde{K}_h \underline{y}_\varrho = 0$.
When introducing the discrete Lagrange multiplier
\[
  \underline{\lambda} :=
  M_h \underline{y}_\varrho + \varrho K_h \underline{y}_\varrho +
  \widetilde{K}_h^\top \underline{p}_\varrho - \underline{f} \in
  {\mathbb{R}}^M,
\]
we conclude a linear system of algebraic equations,
\begin{equation}\label{LGS constraint}
  \left(
    \begin{array}{cc}
      M_h + \varrho K_h & \widetilde{K}_h^\top \\[1mm]
      \widetilde{K}_h &
    \end{array}
  \right)
  \left(
    \begin{array}{c}
      \underline{y}_\varrho \\[1mm]
      \underline{p}_\varrho
    \end{array}
  \right)
  =
  \left(
    \begin{array}{c}
      \underline{f} + \underline{\lambda} \\[1mm]
      \underline{0}
    \end{array}
  \right),
\end{equation}
where the discrete complementarity conditions can be written as,
for $k=1,\ldots,M$ and $c>0$,
\begin{equation}\label{complementarity}
  \lambda_k = \min \{ 0 , \lambda_k + c (g_+-y_{\varrho,k}) \}
  + \max \{ 0, \lambda_k + c (g_--y_{\varrho,k}) \} .
\end{equation}
For the solution of \eqref{LGS constraint} and \eqref{complementarity}
we can apply a semi-smooth Newton method which is equivalent to a
primal-dual active set strategy, see, e.g.,
\cite{Hintermueller2003}, and \cite{GanglLoescherSteinbach:2025}.

%
%

\section{Numerical results}\label{Sec:NumericalResults}
We start our numerical experiments with a smooth harmonic target
$\overline{y}$ defined in two-dimensional (2D) domains $\Omega$ with
different regularity properties. In particular, we want to verify our
theoretical results in the case of domains $\Omega$ with a piecewise
smooth boundary. While the theory for piecewise smooth boundaries was
only presented for 2D domains, we study the 3D case with piecewise smooth
boundary numerically where we consider the unit cube $\Omega = (0,1)^3$
as computational domain.
In Subsection~\ref{Subsec:NumericalResults:HarmonicTarget} we first consider
a harmonic target $\overline{y}$, then a non-harmonic target in
Subsection~\ref{Subsec:NumericalResults:L2OrthogonalSplitting},
and finally, in Subsection \ref{Subsec:NumericalResults:Constraints}, 
we again consider the harmonic target from
Subsection~\ref{Subsec:NumericalResults:HarmonicTarget}, but now  we impose
box constraints on the control $u_\varrho = \gamma_0^{int} y_\varrho$ 
on the boundary $\Gamma$. In all 3D examples, we solve the dual Schur
complement system \eqref{Eqn:DualSchurComplementSystem} by means of the PCG 
(SC-PCG in Tables \ref{tab:solver_dirichlet_bnd_harmonic_target} and
\ref{tab:solver_dirichlet_bnd_L2OrthogonalSplitting_target}) 
with the preconditioner $\widetilde{\mathbf{C}}_h$,
and compare it with the plain CG 
(SC-CG in Tables \ref{tab:solver_dirichlet_bnd_harmonic_target} and \ref{tab:solver_dirichlet_bnd_L2OrthogonalSplitting_target}).
The preconditioner $\widetilde{\mathbf{C}}_h$ is the inexact version of 
the preconditioner $\mathbf{C}_h = \mathbf{K}_{0h}^2$
replacing $\mathbf{K}_{0h}$ by Ruge-St\"uben's AMG preconditioner; see 
\cite{BraessPeisker:1986IMAJNA} for the use of multigrid preconditioners
in squared matrices, 
\cite{JungLangerMeyerQueckSchneider1989MGS3} for multigrid preconditioner
in general, and \cite{RugeStueben1987SIAM}  for Algebraic Multigrid \`{a}
la Ruge and St\"uben. The PCG and CG iterations are stopped as soon as the
Euclidean norm of the residual is reduced by the factor $10^{-8}$.

\subsection{Numerical illustration of the theoretical results in 2D}
\label{Subsec:NumericalResults:DomainRegularity}
Let us consider the smooth and harmonic target 
\begin{equation}\label{Eqn:Examples:harmonic-smooth-2D}
  \overline{y}(x) = x_1^3-3x_1x_2^2 \in \mathcal{C}^\infty (\overline{\Omega}),
\end{equation}
and three different domains: the unit disc
$\Omega_{\text{disc}} :=\{ x \in\mathbb{R}^2:\, |x| <1\}$
as a domain with smooth boundary, the unit square
$\Omega_{\text{quad}} := (0,1)^2$ with a piecewise smooth boundary, and the
L-shaped domain $\Omega_{\text{L}} := (-1,1)^2\setminus
\overline{\Omega}_{\text{quad}}$,
which has a piecewise smooth boundary, but is not convex. The computed errors
of the numerical solutions are depicted in
Figure \ref{Fig:Examples:errors in 2D on different domains}, where we observe
an optimal order of convergence for the smoothly bounded domain
$\Omega_{\text{disc}}$ for $\varrho = h^2$, while we see the logarithmic
dependency on the mesh size for the piecewise smooth domains
$\Omega_{\text{quad}}$ and $\Omega_{\text{L}}$, respectively. We regain
optimal rates when choosing $\varrho = h^2/|\log(h)|$ as predicted by
the theory, see Theorem \ref{Theorem:DiscretizationErrorEstimate}.
Noteworthy, the non-convexity of the domain $\Omega_L$ does not affect
the rate of convergence.

\begin{figure}[htpb!]
  \centering
  \begin{tikzpicture}[scale=0.725,transform shape]
    \begin{axis}[
        xmode = log,
        ymode = log,
        xlabel=$N$,
        ylabel=$\| y_{\varrho h}- \overline{y} \|_{L^2(\Omega)}$,
        legend pos=south west,
        legend style={font=\tiny}]
      
      \addplot [solid, mark=o, color=blue] table [col sep=
        &, y=errL2, x=N]{tab_y1-rhoh2-circle.dat};
      \addlegendentry{$\Omega_{\text{disc}}$, $\varrho = h^2$}
      \addplot [mark=square, color=red] table [col sep=
        &, y=errL2, x=N]{tab_y1-rhoh2-rect.dat};
      \addlegendentry{$\Omega_{\text{quad}}$, $\varrho = h^2$}
      \addplot [mark=text,
      mark options={text mark=\textsf{L}, text mark as node=true, scale=0.8}, color=brown] table [col sep=
        &, y=errL2, x=N]{tab_y1-rhoh2-Lshape.dat};
      \addlegendentry{$\Omega_{\text{L}}$, $\varrho = h^2$}
      
      \addplot[domain = 10^3:10^6, black, samples = 2] {110*1/x};
      \addlegendentry{$h^{2}$}
      \addplot[domain = 10^3:10^6, dashed, black, samples = 2] {7*1/x*(1+1/2*abs(ln(x)))};
      \addlegendentry{$h^{2}(1+|\log(h)|)$}
      \addplot[domain = 10^3:10^6, dashed, black, samples = 2] {30*1/x*(1+1/2*abs(ln(x)))};
    \end{axis}
  \end{tikzpicture}
  \hfill 
  \begin{tikzpicture}[scale=0.725,transform shape]
    \begin{axis}[
        xmode = log,
        ymode = log,
        xlabel=$N$,
        ylabel=$\| y_{\varrho h}- \overline{y} \|_{L^2(\Omega)}$,
        legend pos=south west,
        legend style={font=\tiny}]
      
      \addplot [solid, mark=square*, color=red] table [col sep=
        &, y=errL2, x=N]{tab_y1-rhoh2logh-rect.dat};
      \addlegendentry{$\Omega_{\text{quad}}$, $\varrho = \frac{h^2}{|\log(h)|}$}
      \addplot [solid, mark=text,
      mark options={text mark=\textsf{L}, text mark as node=true, scale=0.8}, color=brown] table [col sep=
        &, y=errL2, x=N]{tab_y1-rhoh2logh-Lshape.dat};
      \addlegendentry{$\Omega_{\text{L}}$, $\varrho = \frac{h^2}{|\log(h)|}$}

      \addplot[domain = 10^3:10^6, black, samples = 2] {10*1/x};
      \addplot[domain = 10^3:10^6, black, samples = 2] {50*1/x};
      \addlegendentry{$h^{2}$}
    \end{axis}
  \end{tikzpicture}
  
  \caption{The error $\norm{y_{\varrho h}-\overline{y}}_{L^2(\Omega)}$ for the
    smooth and harmonic target \eqref{Eqn:Examples:harmonic-smooth-2D} on
    domains with different regularity. 
    Here, the log denotes the natural logarithm.}
  \label{Fig:Examples:errors in 2D on different domains}
\end{figure}
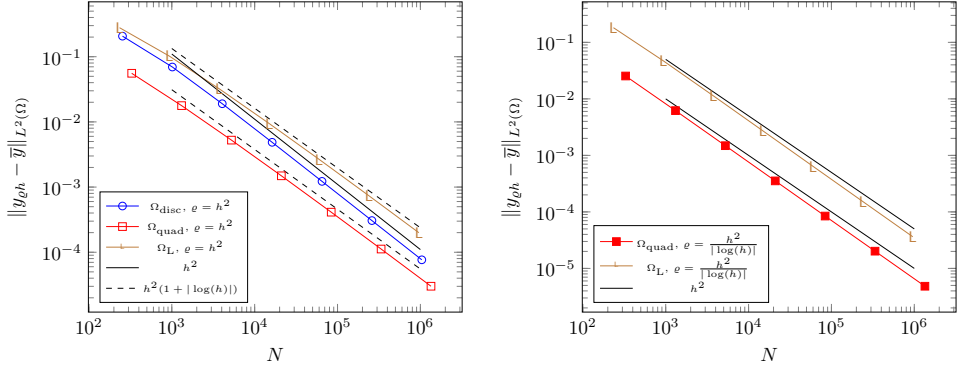

\subsection{Harmonic target in 3D}
\label{Subsec:NumericalResults:HarmonicTarget}
The target is given by the smooth harmonic function
\begin{equation}
  \label{Eqn:HarmonicTarget:Definition}
  \overline{y}(x) = x_1^2- \frac{1}{2} \, x_2^2 - \frac{1}{2} \, x_3^2
  \quad \mbox{for} \; x \in \Omega = (0,1)^3.
\end{equation}
The numerical results are shown in
Table~\ref{tab:solver_dirichlet_bnd_harmonic_target}, where we have chosen
$\varrho = h^2$ with the mesh size $h=2^{-(L+1)}$,
as in the case of a smooth boundary. As in the 2D case
we observe a reduced order of convergence. Concerning the iterative
solution we observe that the preconditioner improves the convergence of
the CG considerably, but the number of iterations still mildly depends on
$h$ as predicted in Section \ref{Sec:Solver}.

\begin{table}[ht]
  {\small
      \begin{center}
    \begin{tabular}{llcccc}
      \hline
      $L$ & $M$ & $\|\overline{y} - y_{\varrho h}\|_{L^2(\Omega)}$
      & eoc & \#CG Its & \#PCG Its\\
      \hline
      $1$&$125$&$1.61$e$-1$& $-$  &$16$  & $13$\\
      $2$&$729$&$6.91$e$-2$& $1.22$  &$43$  & $24$\\
      $3$&$4,913$&$2.50$e$-2$& $1.47$ & $115$ & $33$\\
      $4$&$35,937$&$8.18$e$-3$& $1.61$ & $333$ &$46$ \\
      $5$&$274,625$&$2.50$e$-3$&$1.71$ &$1159$ &$60$ \\
      $6$&$2,146,689$&$7.29$e$-4$&$1.78$ & $4000$ &$86$\\
      $7$&$16,974,593$&$2.06$e$-4$&$1.82$ & $-$ &$133$\\
      \hline
    \end{tabular}
    \end{center}
    \caption{Harmonic target \eqref{Eqn:HarmonicTarget:Definition}:
    error and experimental order of convergence (eoc), number
    of CG and PCG iterations for Schur complement system.} 
    \label{tab:solver_dirichlet_bnd_harmonic_target}
  }
\end{table}
\subsection{Non-harmonic target in 3D}
\label{Subsec:NumericalResults:L2OrthogonalSplitting}
The target is now given by the non-harmonic function
\begin{equation}
  \label{Eqn:L2OrthogonalSplitting:Definition}
  \overline{y}(x) = \underbrace{x_1^2-\frac{1}{2} \, x_2^2-
    \frac{1}{2} \, x_3^2}_{=: \, \overline{y}_1(x)}+
  \underbrace{
    \Delta\prod_{i=1}^3\left(
    x_i^2(1-x_i)^2
    \right)}_{=: \, \overline{y}_2(x)} \quad \mbox{for}
    \; x  \in \Omega = (0,1)^3,
\end{equation}
that is nothing but an $L^2$ orthogonal direct sum of the harmonic part 
$\overline{y}_1 \in \mathcal{H}(\Omega)$, already considered in 
Subsection~\ref{Subsec:NumericalResults:HarmonicTarget}, 
and the non-harmonic part $ \overline{y}_2 \in \Delta (H^2_0(\Omega))$.
The numerical results are shown in Table~\ref{tab:solver_dirichlet_bnd_L2OrthogonalSplitting_target}.
We again choose $\varrho = h^2$ as in the case of a smooth boundary. 
We observe that the computed finite element state $y_{\varrho h}$ converges to 
the harmonic part $\overline{y}_1$ with the same $L^2$ rates as
observed in Subsection~\ref{Subsec:NumericalResults:HarmonicTarget}; 
cf. with Table~\ref{tab:solver_dirichlet_bnd_harmonic_target}.
The iteration numbers of the SC-CG and SC-PCG behave similarly as for the 
harmonic target.

\begin{table}[ht]
  {\small
      \begin{center}
    \begin{tabular}{llccccc}
      \hline
      $L$ & $M$ & $\|\overline{y} - y_{\varrho h}\|_{L^2(\Omega)}$
      &$\|\overline{y}_1 - y_{\varrho h}\|_{L^2(\Omega)}$
      & eoc &\#CG Its & \#PCG Its\\
      \hline
      $1$&$125$&$1.61$e$-1$ & $1.61$e$-1$ & $-$ &$16$  & $13$\\
      $2$&$729$&$6.92$e$-2$&$6.91$e$-2$& $1.22$  &$43$  & $24$\\
      $3$&$4,913$&$2.52$e$-2$&$2.50$e$-2$ & $1.47$ & $115$ & $33$\\
      $4$&$35,937$&$8.74$e$-3$& $8.18$e$-3$ & $1.61$ & $336$ &$46$ \\
      $5$&$274,625$&$3.97$e$-3$&$2.50$e$-3$ & $1.71$  &$1169$ &$63$ \\
      $6$&$2,146,689$&$3.17$e$-3$ & $7.29$e$-4$ & $1.78$ & $3997$ &$92$\\
      $7$&$16,974,593$&$3.09$e$-3$ & $2.06$e$-4$ & $1.82$ & $-$ &$141$\\
      \hline
    \end{tabular}
    \end{center}
    \caption{$L^2$ orthogonal splitting of the target
      \eqref{Eqn:L2OrthogonalSplitting:Definition}:  
      error of the state $y_{\varrho h}$ with respect to the target
      $\overline{y}$, and its harmonic part $\overline{y}_1$,
    and number
    of CG and PCG iterations for Schur complement system.}
    \label{tab:solver_dirichlet_bnd_L2OrthogonalSplitting_target}
    }
\end{table}

\subsection{Constraints}
\label{Subsec:NumericalResults:Constraints}
We again consider the harmonic target \eqref{Eqn:HarmonicTarget:Definition},
but now we impose box constraints on the control
$u_\varrho = \gamma_0^{int} y_\varrho$ on the boundary $\Gamma$ 
with the lower bound $g_- = -0.7$ and the upper bound $g_+ = + 0.7$;
cf. Section~\ref{Sec:Constraints}.
In Table~\ref{tab:PDAS_Control_Constraints}, we provide the numerical results 
obtained by  the  primal-dual active set (PDAS) method 
for solving \eqref{LGS constraint} and \eqref{complementarity}.
We observe that the $L^2$
distance  $\|\overline{y} - y_{\varrho h}\|_{L^2(\Omega)}$  of the computed
finite element state $y_{\varrho h}$ to the target $\overline{y}$
stagnates as soon as the projection of the target to $K$ is reached, 
where we choose $\varrho = h^2$. The PDAS method stops the iterations
when the points in the active and inactive sets 
do not change anymore. In Figure \ref{fig:PDAS_Control_Constraints},
we visualize the  boundary control on three boundary parts
of the unit cube $\Omega = (0,1)^3$, and
near both the upper and lower bounds of constraints.

\begin{table}[ht]
  {\small
    \begin{center}
      \begin{tabular}{llccl}
      \hline
        $L$ & $M$ & $\|\overline{y} - y_{\varrho h}\|_{L^2(\Omega)}$
        &  \#PDAS & \#Changing points\\
        \hline
        $1$&$125$&$1.61$e$-1$  &$1$  & $\{0\}$\\
        $2$&$729$&$7.15$e$-2$ &$2$  & $\{21;\, 0\}$\\
        $3$&$4,913$&$4.06$e$-2$ & $3$ & $\{200;\, 1;\, 0\}$\\
        $4$&$35,937$&$3.48$e$-2$ & $4$ &$\{923;\, 113;\, 1;\, 0\}$ \\
        $5$&$274,625$&$3.48$e$-2$&$4$ &$\{3,847;\, 741;\, 64;\, 0\}$ \\
        $6$&$2,146,689$&$3.40$e$-2$ & $5$ &$\{15,503;\, 3,097;\, 985;\, 22;\, 0\}$\\
        \hline
      \end{tabular}
    \end{center}
    \caption{Box constraints on the Dirichlet data of the harmonic
      target \eqref{Eqn:HarmonicTarget:Definition}: error, number of
      PDAS iterations, and number of changing points per iteration.}
    \label{tab:PDAS_Control_Constraints}
    }
\end{table}

\begin{figure}[ht]
  \centering
  \includegraphics[width=0.3\textwidth]{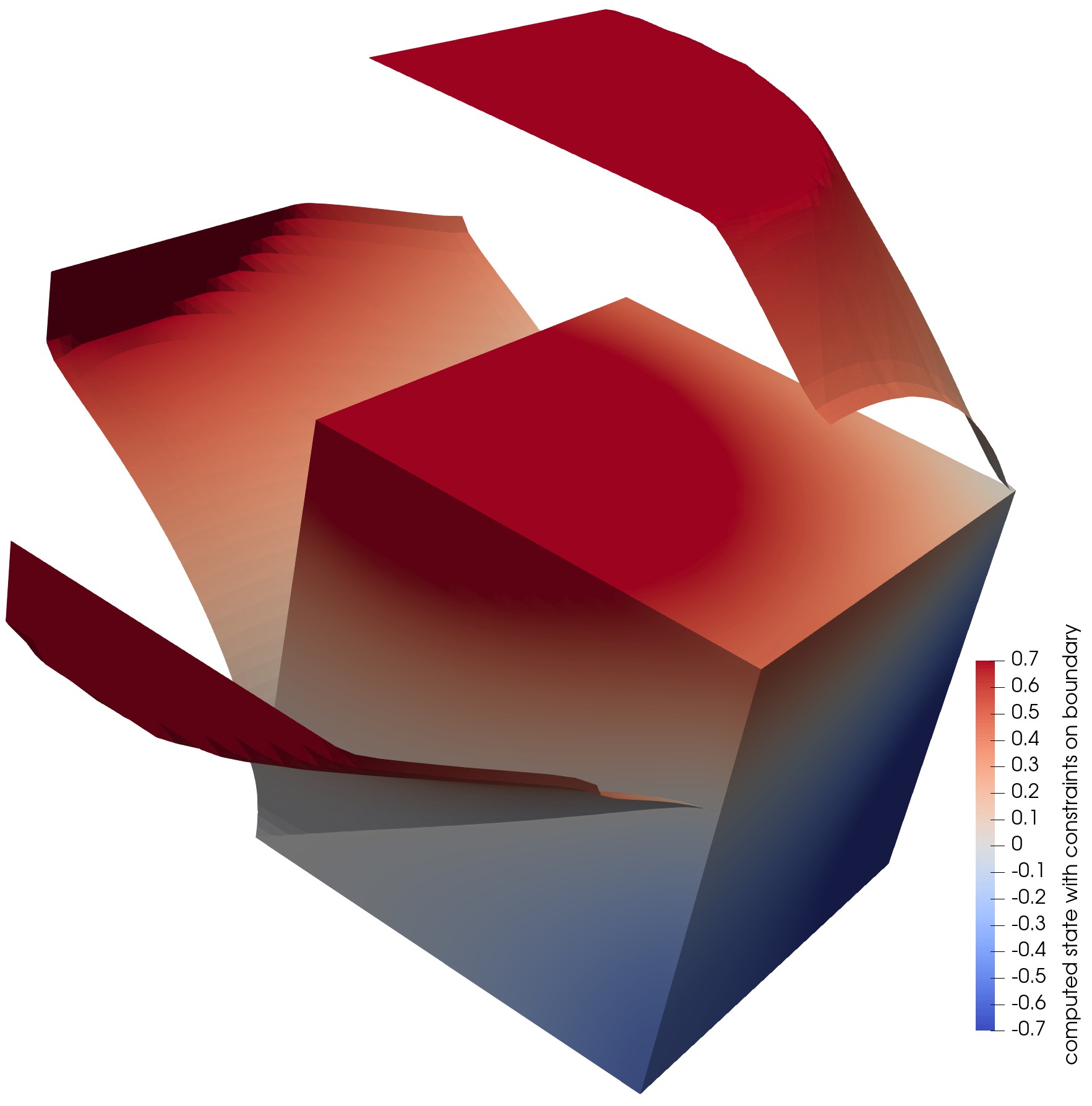}
  \includegraphics[width=0.3\textwidth]{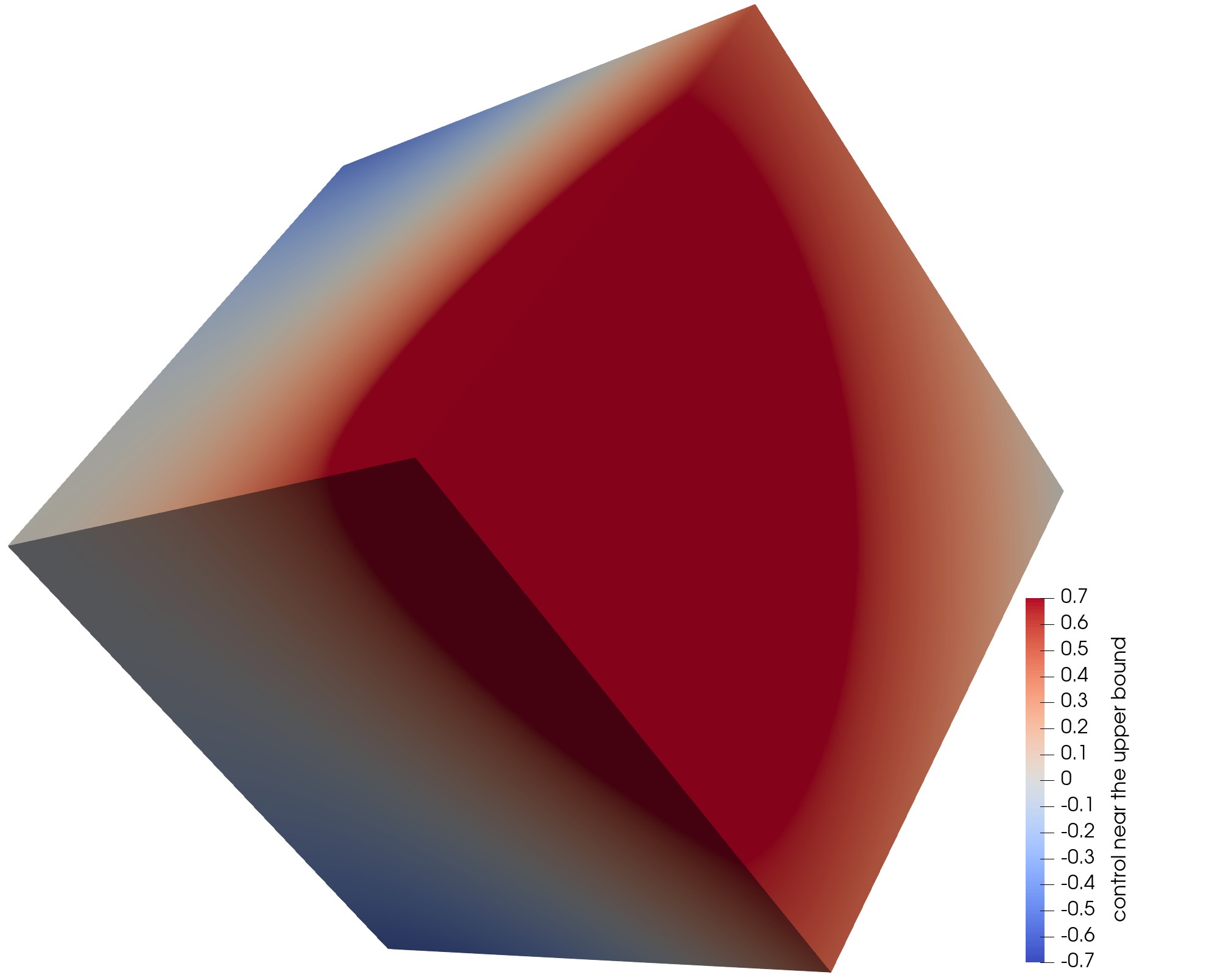}
  \includegraphics[width=0.3\textwidth]{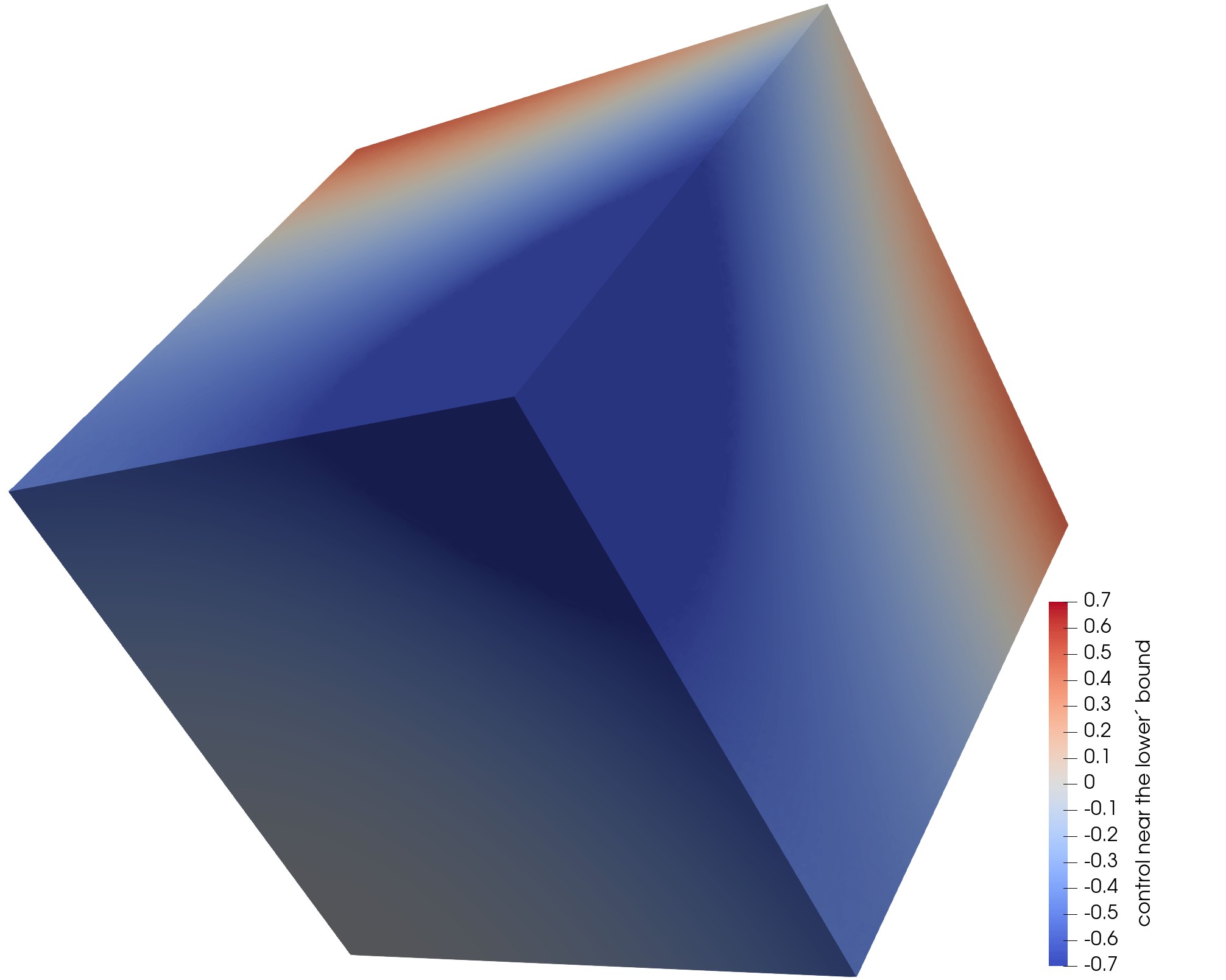}
  \caption{Visualization of the boundary control on three boundary parts
    (left), and near the upper (middle) and lower (right) bounds at level $6$.}
  \label{fig:PDAS_Control_Constraints}
\end{figure}

%
%

\section{Conclusions and outlook}\label{Sec:ConclusionAndOutlook}
We have presented a new approach to the numerical solution of optimal
control problems for the Laplace equation where the Dirichlet boundary data
serves as control. We have used energy regularization in $H^{1/2}(\Gamma)$
that is natural for Dirichlet data when considering the standard
variational formulation of elliptic second order partial differential
equations. We have reduced the optimality conditions to a variational
equality for computing the state in the space of harmonic functions. For a
practical realization, we have introduced a Lagrange multiplier leading to
a mixed variational formulation in $H^1(\Omega) \times H^1_0(\Omega)$ 
that can easily be discretized by standard $C^0$ conforming finite elements.
The analysis of the error in terms of the cost parameter $\varrho$ and the
finite element mesh size $h$ implies that both parameters should appropriately
be balanced in order to obtain asymptotically optimal convergence rates. It
turns out that $\varrho = h^2$ in the case of 2D and 3D domains $\Omega$
with smooth boundaries $\Gamma$. If the boundary $\Gamma$ is only piecewise
smooth, then the choice $\varrho = h^2$ is only suboptimal
for smooth targets as it was shown
theoretically in 2D, and numerically in 3D. Indeed, in 2D, it turns out
that $\varrho = h^2 / |\log(h)|$ is the optimal choice. Furthermore, we have
proposed efficient solvers for the resulting system of finite element
equations, and we have shown that it is easy to incorporate box constraints 
on the control and the state at once. 
 
The precise analysis of smooth targets in
3D domains $\Omega$ with piecewise smooth boundaries, 
and the construction and implementation of asymptotically optimal solvers 
are interesting future research topics. It is clear that the Laplace 
operator plays the role of a model operator that can be replaced 
by other second-order elliptic partial differential operators.
These and other optimal control problems fit into the abstract 
framework recently presented in
\cite{LangerLoescherSteinbachYang2025arXiv2505.19062}.
We only mention here Neumann \cite{BrennerSung2025RAM}
and Robin boundary control,
boundary and initial data control of parabolic initial-boundary value
problems, partial distributed control, and targets given on a subset
or even only on the boundary of the computational domain.
In terms of inverse problems, the latter is also called 
partial or limited observation \cite{MardalNielsenNordaas2017BIT}.

%
%

\section*{Acknowledgments}
We would like to thank the computing resource support of
the high performance computing cluster
Radon1\footnote{https://www.oeaw.ac.at/ricam/hpc} from Johann Radon Institute
for Computational and Applied Mathematics (RICAM). The financial support for the fourth author by the
Austrian Federal Ministry for Digital and Economic Affairs, the National
Foundation for Research, Technology and Development, and the Christian
Doppler Research Association is gratefully acknowledged.


\bibliography{LLSY2025DOCP}
\bibliographystyle{abbrv} 
  

\end{document}